\documentclass[a4paper,14pt]{amsart}
\pdfoutput=1
\usepackage{amsfonts}
\usepackage{amsmath}
\usepackage{amsthm}
\usepackage{amssymb}
\usepackage{caption}
\usepackage{subcaption}
\usepackage{booktabs}
\usepackage{mathtools}
\mathtoolsset{showonlyrefs}

\newcommand{\hidethis}[1]{}

\newcommand{\real}{\mathbb{R}}
\newcommand{\comp}{\mathbb{C}}

\newcommand{\nats}{\mathbb{N}}

\newcommand{\lap}{\Delta}

\newcommand{\oh}{\tfrac{1}{2}}

\newcommand{\fixr}{\big|_r}
\newcommand{\fixz}{\big|_\zeta}
\newcommand{\fixx}{\big|_x}
\newcommand{\fixs}{\big|_s}
\newcommand{\fixsig}{\big|_{\sigma}}
\newcommand{\fixrho}{\big|_\rho}
\newcommand{\fixxi}{\big|_\xi}

\newcommand{\tildehack}{\tilde}

\newcommand{\dopv}{\mathcal{D}^{(v)}}
\newcommand{\doph}{\mathcal{D}^{(h)}}

\newcommand{\fop}{\mathcal{F}}
\newcommand{\foph}{\mathcal{F}^{(h)}}
\newcommand{\fopv}{\mathcal{F}^{(v)}}

\newcommand{\goph}{\mathcal{G}^{(h)}}
\newcommand{\gopv}{\mathcal{G}^{(v)}}

\newcommand{\lop}{\mathcal{L}}
\newcommand{\qop}{\mathcal{Q}}
\newcommand{\rop}{\mathcal{R}}

\newcommand{\ary}{\mathcal{A}}
\newcommand{\bry}{\mathcal{B}}
\newcommand{\cry}{\mathcal{C}}
\newcommand{\north}{N}

\newcommand{\rest}{M}

\newcommand{\sprime}{s'}
\newcommand{\lphi}{\log(\phi)}

\newcommand{\Tcross}{T^{\dagger}}

\newcommand{\defeq}{\vcentcolon=}
\newcommand{\eqdef}{=\vcentcolon}

\theoremstyle{plain}
\newtheorem{theorem}{Theorem}
\newtheorem{lemma}[theorem]{Lemma}

\newtheorem{assumption}[theorem]{Assumption}
\newtheorem{definition}[theorem]{Definition}

\theoremstyle{remark}
\newtheorem*{remark}{Remark}
\newtheorem{example}[theorem]{Example}

\DeclareMathOperator{\Rc}{Rc}
\DeclareMathOperator{\Rm}{Rm}

\DeclareMathOperator{\dist}{dist}

\title{Ricci flow recovering from pinched discs} 
\author{Tim Carson}
\thanks{The author has been partially supported by the NSF grant DMS 1148490.}
\address[Timothy Carson]{University of Texas, Department of Mathematics\\
  2515 Speedway stop C1200\\
  Austin, Texas 78712-1202}

\makeatletter
\newcommand*{\centerfloat}{%
  \parindent \z@
  \leftskip \z@ \@plus 1fil \@minus \textwidth
  \rightskip\leftskip
  \parfillskip \z@skip}
\makeatother

\begin{document}

\begin{abstract}
  We construct smooth solutions to Ricci flow starting from a class of singular metrics and give asymptotics for the forward evolution.  The singular metrics heal with a set of points (of codimension at least three) coming out of the singular point.   We conjecture that these metrics arise as final-time limits of Ricci flow encountering a Type-I singularity modeled on $\real^{p+1} \times S^q$.  This gives a picture of Ricci flow through a singularity, in which a neighborhood of the manifold changes topology from $D^{p+1} \times S^{q}$ to $S^p \times D^{q+1}$ (through the cone over $S^p \times S^q$.)  

  We work in the class of doubly-warped product metrics.  We also briefly discuss some possible smooth and non-smooth forward evolutions from other singular initial data.
\end{abstract}
\maketitle

%\section{Introduction}

\subsection{Description of a flow through a singularity}
In this paper we prove the existence of, and provide asymptotics for, Ricci flow starting from certain singular initial metrics.  We believe that these singular initial metrics can arise as limits of Ricci flow with a finite-time singularity, in a way described by Hamilton in \cite{h_formation}.  We also provide formal arguments for the asymptotics of the flow into the singularity.

The flow through the singularity is illustrated in Figure \ref{figure:full_flow}.  Take $(N^{p+1}, g_N)$ to be a manifold with small curvature, and now consider a metric on $N \times S^q$ ($q \geq 2$) obtained by placing a sphere of some size on every point in $N$.  After running Ricci flow for a short time, the spheres' sizes obey a reaction-diffusion equation while the metric on $N$ also changes.  If we make the initial spheres on $N$ small in some region, we can force a singularity to occur in that region, at one point on $N$.  The singularity will be modeled on $R^{p+1} \times S^{q}$ if the $N^{p+1}$ factor stays relatively smooth. Near the singularity point but before the singular time the manifold has a region with topology $D^{p+1} \times S^q$.  At the singular time this region takes on the topology of the cone over $S^p \times S^q$ (but is not close to a metric cone). 

Afterwards, the $S^p$ factor grows, and the manifold takes on the topology of $S^p \times D^{q+1}$ near the singularity.  The case $p=0$ is the well known standard neckpinch case, see Figure \ref{np_basic}. A somewhat disappointing aspect of this topological change is that if $p \geq 2$ then Ricci flow can perform the opposite topological change (by reversing the roles of $p$ and $q$ in the description above.)

\begin{figure}[tp]
 \centerfloat
  \includegraphics[width=1.4\textwidth]{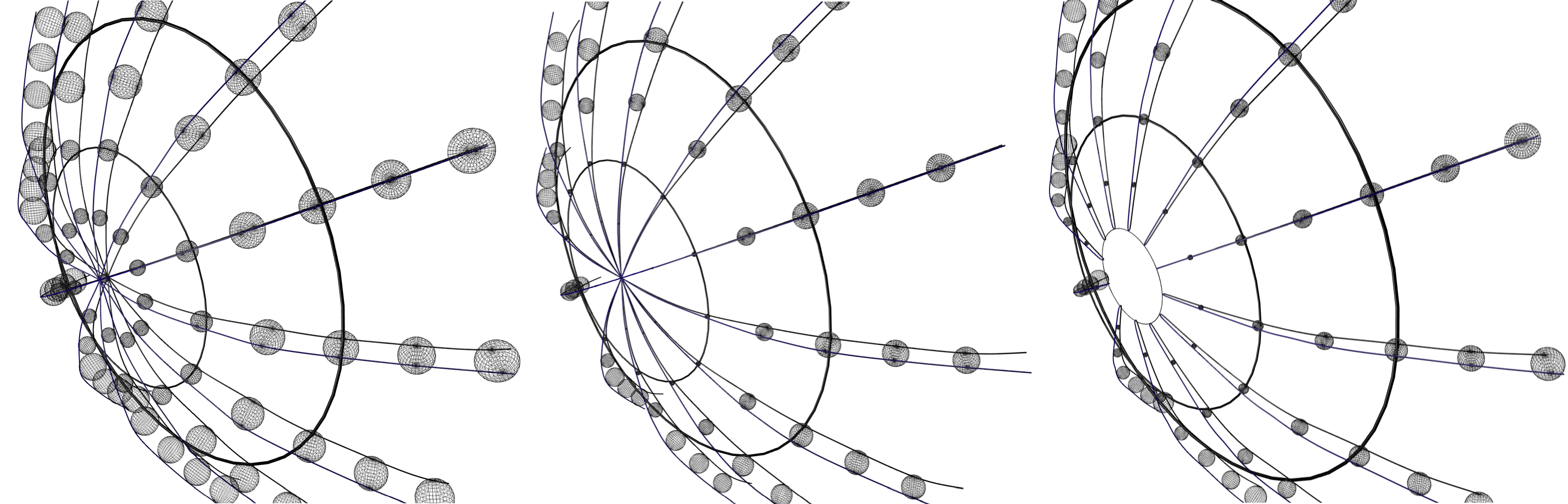}
  \caption{A neighborhood of a doubly warped product undergoing Ricci flow.  Initially, the neighborhood has topology $D^2 \times S^2$.  The size of the $S^2$ at the center of the $D^2$ is quite small.  In the middle, the flow has reached a singularity.  The neighborhood has the topology of the cone over $S^1 \times S^2$.  The flow continues on the right. The manifold is again smooth and the neighborhood has topology $S^1 \times D^3$.  The size of the inner circle is relatively large compared to the curvature of the $S^2$ factor near the inside.  In a rescaling going backwards in time the manifold  will converge to $(\real \times \text{Steady Bryant Soliton})$.  }
  \label{figure:full_flow}
\end{figure}

\begin{figure}[ht]
  \centering
  \includegraphics[scale=.4]{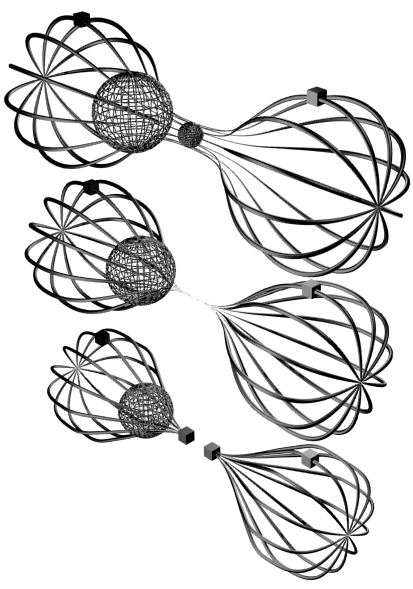}
  \caption{
    A doubly warped product over an interval with fibers $S^0$ and $S^2$ undergoing Ricci flow.  Some $S^0$'s are marked with pairs of cubes.  Having an $S^0$ fiber is somewhat equivalent to asking for the size of the $S^2$ in a singly warped product to be an even function.  As far as intrinsic metric spaces are concerned there is no size of the $S^0$ factor (i.e. if $p=0$ the function $\psi$ has no effect.)  If we were considering mean curvature flow, the size of the $S^0$ factor could be understood as the distance between the two disconnected parts after a neckpinch.    }
  \label{np_basic}
\end{figure}

\subsection{Forward evolutions from singular metrics}
Our work continues the investigation into which singular metrics have a forward Ricci flow which smooths them.  A usual method to construct such flows (which we use) is to consider the Ricci flow from smooth metrics approximating the initial metric.  If these smooth mollifications of the initial metric have uniform short time existence, then one can extract a limiting flow as the modification goes to zero.

We work within the class of doubly warped product metrics, which simplifies the Ricci flow equation to a system of parabolic equations on an interval.  The authors of \cite{ACK} and \cite{recover} similarly construct flows starting from singular initial data, but in the case of a singly warped product over an interval.  As with our work, the motivation was to understand certain flows through singularities.  The singular initial metrics considered there arise from singularities modeled on $\real \times S^q$ which are type I and type II, respectively.

In three dimensions, Ricci flow through singularities is now well understood.  Kleiner and Lott \cite{Singular} show that Perelman's Ricci flow with surgery \cite{Perelman2} can be seen as an approximation to a true flow with some singularities occurring and healing.  One result of \cite{Singular} is that such a flow through singularities in three dimensions has finitely many ``bad worldlines'', so intuitively speaking only finitely many points come out of singularities.  (E.g. the two cubes marking the inner $S^0$ in the third part of Figure \ref{np_basic}.)  It is unsurprising that this should be false in higher dimensions, since the proof comes from a classification of possible ways a singularity can heal. To our knowledge we can hope for a statement about the codimension of such points.  (In our examples they have codimension three or more.)

In \cite{conicalSing}, Gianniotis and Schulze construct flows starting from singular metrics without any symmetry assumptions.  Their initial metrics have isolated singular points modeled on  non-negatively curved cones over $S^n$.  These are conjectured to arise as final time limits of type-I singularities, when the shrinking soliton which the singularity is modeled on is asymptotically Ricci flat.  The metrics immediately heal, using expanding solitons which are asymptotic to the cones (shown to exist by Deruelle \cite{conical_expanders}).  

The examples from \cite{conicalSing} heal with an expander with scale $\sqrt{t}$, whereas the examples here, in \cite{ACK}, and in \cite{recover} heal at a smaller scale.  In all of these examples if for the initial metric, $|\Rm|^2 = O(\gamma(d); d \searrow 0)$ where $d = \dist(\text{Singular set}, p)$, then for the forward evolution $|\Rm|^2 = O(\gamma(\sqrt{t}); t \searrow 0)$.  This can be guessed, for example, by applying Hamilton's Harnack inequality (Corollary 3.1 of \cite{hamharnack}) without bothering to verify the assumptions.  The evolution we consider here has an additional scale at which the geometry is changing: the inner circle on the right side of Figure \ref{figure:full_flow} is asymptotically larger than the scale at which the singularity is healing.

We show that there is a flow $g(t)$ emerging from a singular initial metric $g_{init}$ in the following sense.  (This is identical to the definition in Theorem 1.1 of \cite{conicalSing}.)
\begin{definition}\label{smooth_complete_evo}
  Let $(M, g_{init})$ be a Riemannian manifold, which is not complete and whose metric space completion is not a smooth manifold.  We say $(M, g(t))_{t \in [0, T)}$ is a smooth complete Ricci flow emerging from $(M, g_{init})$ if
\begin{itemize}
  \item $g(0) = g_{init}$.
  \item For $t > 0$, the metric space completion of $(M, g(t))$ is a smooth Riemannian manifold $(\bar M, \bar g(t))$.
  \item $(\bar M, \bar g(t))$ is a solution of Ricci flow on $(0, T)$ and $(M, g(t))$ is a solution of Ricci flow on $[0, T)$.
  \item As $t \searrow 0$, the smooth spaces $(\bar M, \bar g(t))$ approach the metric space closure of $(M, g(t))$ in the Gromov-Hausdorff sense.
  \end{itemize}
\end{definition}
Definitions of weak Ricci flows have been emerging recently, but there is no general existence which is strong enough to deal with our situation.  In \cite{weakhaslhofernaber}, Haslhofer and Naber characterize Ricci flow in terms of analytic estimates on path space.
Sturm also provides a characterization of Ricci flow in \cite{weaksturm} which makes sense for metric-measure spaces.  This characterization is stated for a family of metric-measure spaces on a \emph{fixed} topology, so some further work is needed to deal with the case we consider here.  In \cite{weaksturm2}, Kopfer and Sturm provide estimates for this definition, but under the assumption that distances are log-lipschitz in time.  This is certainly impossible in our case because certain distances go to zero in finite backward time.

While a forward evolution satisfying Definition \ref{smooth_complete_evo} is nice, the following example shows that this is not something we can always demand.
\begin{example}\label{example:scarring}
  Consider first the example where a Riemannian manifold is a warped
  product of $S^4$ over a circle $S^1$, which at some point sees a neckpinch.
  The manifold can recover by becoming an $S^5$. Away from two poles
  the manifold will continue to be a warped product of
  $S^4$ over an interval.  Now consider instead the situation where we replace the
  initial metric's $S^4$ fiber, at each point on the $S^1$, with an
  $S^2 \times S^2$ with the standard metric and the same scalar curvature as the original
  $S^4$.  Since $S^2 \times S^2$ is an Einstein manifold, according to equations \eqref{rc_base} and \eqref{rc_fiber}
  below, the Ricci curvature essentially does not care that the fiber
  is an $S^2 \times S^2$ instead of an $S^4$, and Ricci flow acts the
  same way up to the singularity.  Also, any weak
  definition of Ricci flow should allow us to continue in the same way as when the
  fiber was $S^4$.  This means that after the singular time, there are
  ``scars'' for a reasonable forward evolution; at the two new
  points there are two conical singularities.  Near these conical
  singularities, the full curvature tensor is unbounded but the Ricci
  curvature is bounded (at any time after the singular time).
\end{example}
Note that this example is unstable because $S^2 \times S^2$ is unstable.  By using, instead of $S^2 \times S^2$, an Einstein manifold such that the Ricci flat cone over it is stable, the example may become stable.  (For example, we may use $\comp P^2$ by \cite{stabRcFlatCone}.)

\subsection{Precise statements}

The metrics we consider are of the following form.  The space has topology $(0,1) \times S^p \times S^q$ outside of a lower dimensional set.  The metrics are doubly warped products over an interval; at each time they are given by
\begin{align}
  \sprime(x)^2 dx^2 + \psi(x)^2 g_{S^p} + \phi(x)^2 g_{S^q}, \label{dwp_form}
\end{align}
where $x$ is the coordinate on $(0,1)$, and $g_{S^p}$, $g_{S^q}$ are the lifts of round metrics on the spheres.  We use the following notation for certain subsets of metrics.
\begin{definition}\label{north}
  Suppose $(M, g)$ has the form \eqref{dwp_form} outside of a lower-dimensional subset.  Define the arclength coordinate by $ds = s'dx$ and $s=0$ at $x=0$, i.e. $s(x) = \int_0^xs'(\tilde x) d\tilde x$.  We call the set $\{s = 0\}$ the tip.  For $s_*$ given let
  \begin{align}
    \label{eq:5}
    \north_{s < s_*} = \{(x, a, b): s(x) < s_*, a \in S^p, b \in S^q\}
  \end{align}
  and let $\north_{r < r_*}$ the connected component of $\{(s, a, b): \phi(s) < r_* \}$ which borders the tip, when this is well defined.  Let
  \begin{align}
    \rest_{s > s_*} = M - \north_{s < s_*}, \qquad
    \rest_{r > r_*} = M - \north_{r < r_*}.
  \end{align}
\end{definition}
Our existence theorem precisely is as follows.
\setcounter{topnumber}{0}
\begin{theorem}\label{main_theorem}
  Suppose $(M, g_{init})$ satisfies the following properties.
  \begin{enumerate}
  \item
    The manifold $M$ has the topology $ (0, L) \times S^p \times S^q$, and the metric has the form
    $$g_{init} = ds^2 + \psi_{init}^2 g_{S^p} + \phi_{init}^2 g_{S^q},$$
    where $\psi_{init}$ and $\phi_{init}$ are functions of $s$ and $g_{S^p}$ and $g_{S^q}$ are lifts of the round metrics of radius one.  The length $L>0$ may be $\infty$.
  \item\label{grad_bnd_assume}
    The warping functions $\phi_{init}$ and $\psi_{init}$ satisfy the gradient bound
    \begin{align}
      (\partial_s \phi_{init})^2 =  |\nabla \phi_{init}|^2 \leq 1, \qquad (\partial_s \psi_{init})^2 \leq 1. \label{grad_bnd}
    \end{align}
  \item \label{rm_bnd}
    For any $s_1$ and $k \in \nats \cup \{0\}$, $|\nabla^k \Rm|$ is bounded in $\rest_{s > s_1}$.
  \item \label{inc_rad}
    There is $s_2>0$ such that $\phi_{init}$ is increasing in $\north_{s < s_2}$.
  \item 
    As $s \searrow 0$,
    \begin{align}
      \begin{split}
        \label{assumed_asymptotics}
      \psi_{init} = s(1 + o(1)), \qquad \phi_{init}(s) = k \frac{s}{\sqrt{|\log s|}} (1 + o(1)) \\
      \partial_s \phi_{init}(s) = k \frac{1}{\sqrt{|\log s|}} (1 + o(1)) 
      \end{split}
    \end{align}
    for some $k \in \real_{>0}$.
  \end{enumerate}
  Then there is a smooth complete Ricci flow $(M, g(t))_{t \in [0, T)}$ emerging from $(M, g_{init})$.  For $t > 0$ the metric space completion of $(\north_{s < s_*}, g(t))$ has the topology $S^p \times D^{q+1}$. The metric continues to satisfy items \ref{grad_bnd_assume} through \ref{inc_rad}. 
\end{theorem}
\setcounter{topnumber}{2}

\begin{remark}
  We can easily relax our assumptions in a few ways, but do not to save
  on notation.  The gradient bound assumption \eqref{grad_bnd} can be
  relaxed to a bound by any $C \geq 1$, because any such bound is
  preserved by Ricci flow.  We could assume the same singular
  asymptotics at $s = L$ as well; then all of our estimates away from
  $s = 0$ would turn into estimates away from both endpoints of the
  interval.  The situation with two singularities happens for example
  with standard neckpinches on $S^1 \times S^p$.  Finally, our proof carries through for some other asymptotic assumptions on the initial data, see Section \ref{conjectured}.
\end{remark}

Our proof also gives asymptotics for the forward evolution.  We split the forward evolution into three regions, which we describe below.  Let $r = \phi$, $v = \phi_s^2$, and $h = \psi^2$, as in Section \ref{r_coordinate_section}.  Then:
\begin{itemize}
\item \textbf{Outer region.}  As $r, t$ go to $(0, 0)$ with $\rho = r/\sqrt{t} \to \infty$ we have
  \begin{align}
    \label{eq:74}
    v &= (1 + o(\rho^{-2}))v_{init},\\
    h &= (1 + o(\rho^{-2}))h_{init}.
  \end{align}
  Strictly away from the tip, the metric is smooth.  The outer region gives an idea of how long the initial metric is relatively unaffected by the high curvature region at the tip.  This region can be derived by applying the pseudolocality theorem to the initial metric at points near the tip.  the scale at which we can apply pseudolocality is $r$, and it gives regularity for time $t < \sqrt{r}$.
\item \textbf{Parabolic region.} As $r, t$ goes to $(0,0)$ with $\rho = \sqrt{\frac{1}{t}}r$ bounded but $\zeta = \sqrt{\frac{|\log(t)|}{t}}r \to  \infty$,
  \begin{align}
    \label{eq:16}
    v &= (1 + o(\zeta)^{-1/2})\frac{2k^2}{|\log t|}(1 + 2(q-1)\rho^{-2}) \\
    \frac{1}{t}h &= (1 + o(\zeta)^{-1/2})\frac{|\log t|}{2k^2}(\rho^2 + 2(q-1))
  \end{align}
  The parabolic region is where the metric looks like a relatively long cylinder.  The radius goes from of $\phi^2 g_{S^q}$ goes from $\sqrt{t}$ to $\sqrt{t/|\log t|}$ in length $c\sqrt{t |\log t|}$.

  In both parabolic and outer regions, $\phi_s^2$ goes to zero as $(r, t)$ goes to $(0,0)$.  As a result, both $\phi_s^2$ and $\psi$ approximately satisfy a linear first-order equation in $r = \phi$.
\item \textbf{Tip region.} As $r,t$ goes to $(0,0)$ with $\zeta = \sqrt{\frac{|\log(t)|}{t}} r$ bounded,
  \begin{align}
    \label{eq:17}
    v &= \left(1 + o\left(\frac{1}{|\log t|}\right)\right)
        \bry\left( \frac{\zeta}{4(q-1)k^2}  \right) \\
    \frac{|\log t|}{t}h
      &= \left(1 + o\left(\frac{1}{|\log t|}\right)\right)
        |\log t|^2
        \frac{q-1}{k^2}
  \end{align}
  Here $\bry:[0, \infty) \to \real$ is a function such that
  \begin{equation}
    \label{eq:88}
    \frac{dr^2}{\bry(r)} + r^2 g_{S^q}
  \end{equation}
  is a Bryant soliton.

  In the parabolic and tip regions together, the metric looks like a standard solution crossed with $\real^p$.    In the tip region, it looks like the steady soliton $(\text{Bryant Soliton}) \times (\text{Flat }\real^p)$, scaled by $\sqrt{t/|\log t|}$.  At the scale $\sqrt{t/|\log t|}$, the curvature of the $\real^p$ factor is $\left|\Rm\big|_{TS^p}\right|  = O(|\log t|^{-1})$.
\end{itemize}

% \subsection{Weak Ricci flows}\label{weak_discussion}

% We use a definition (Definition \ref{smooth_complete_evo}) which does not enforce any condition at the points in space-time where the metric is singular.  One problem is absolutely horrid non-uniqueness in some cases.
% \begin{example}
%   Topping \cite{topping} constructs smooth complete Ricci flows
%   emerging from a flat $\real^2 - \{0\}$ which complete the metric by
%   instantaneously turn $B(0, r)$ into an end of infinite length.  (See
%   also many later works by Topping for Ricci flow starting from many
%   classes of metrics on surfaces.) By our definition a static flat
%   metric on $\real^2 - \{0\}$ is of course also a smooth complete
%   Ricci flow emerging from $\real^2 - \{0\}$.
% \end{example}
% If we want to pick out the static flat $\real^2$, we should be considering complete initial metric spaces.  Note that \cite{ACK} and \cite{recover} show that this ``completion via noncompactness'' cannot happen for initial metrics coming from neckpinches and degenerate neckpinches, respectively.  Theorem \ref{asymptotics_theorem} shows the same in our situation; note that no assumption is made on the topology our boundary conditions.

\subsection{Conjectured flows}\label{conjectured}
In Appendix \ref{before_asymptotics} we formally derive the asymptotics for a flow into the singular metric we consider.  Briefly, the $D^{p+1}$ factor stays relatively flat, which is unsurprising given the stability of flat metrics under Ricci flow.  Then the factor $\phi^2 g_{S^q}$ nearly obeys a reaction-diffusion equation on $\real^{p+1}$.

One may ask about forward evolutions from initial singular metrics with asymptotics besides \eqref{assumed_asymptotics}.  First consider keeping $\psi_{init}(s) = s(1 + o(1))$ and changing $\phi$.  We claim that our proof and construction of barriers can be generalized to this case, assuming the condition that as $s \searrow 0$
\begin{align}
  \phi_{init}(s) = o(s), \quad s \partial_s \phi_{init}(s) = a \phi(s) (1 + o(1)), \quad s^2 \partial^2_s \phi(s) = O(\phi(s)).
\end{align}
Here $a \in \real$ is a constant.  In this case the asymptotics are similar but the Bryant soliton which forms at the tip has its metric scaled by $\phi_{init}(\sqrt{t})^2$.

We may also ask about changing $\psi_{init}$.  For our initial data,
\begin{align}
  \label{eq:18}
  \psi_{init}^2 = \frac{1}{k^2}\phi_{init}^2|\log \phi_{init}|(1 + o(1)).
\end{align}
Our method also can deal with any case when $\psi$ is asymptotically strictly larger than $\sqrt{\frac{p-1}{q-1}}\phi_{init}$.  Here $c_0 = \sqrt{\frac{p-1}{q-1}}$ is the constant so that $c_0^2 g_{S^p} + g_{S^q}$ is an Einstein with $\Rc = (q-1)g$.  The ratio $\frac{\psi_{init}}{c_0 \phi_{init}}$ controls how quickly the $S^p$ factor approaches flat $\real^{p}$.  

An interesting case that our proof \emph{cannot} handle is
\begin{align}
  \label{eq:27}
  \psi_{init} = c_0 \phi_{init} \left( 1 + |\log(\phi_{init})|^{-1} \right).
\end{align}
In this case, formal asymptotics tell us that we should be able to glue in an Ivey soliton at the tip. (This is a steady soliton found by Ivey \cite{ivey} on a doubly warped product of two spheres. It closes off to be a $S^p \times D^{q+1}$ at $s=0$ and at infinity the size of the two factors is asymptotically equal.)  Our proof does not work because we used a loose gradient bound (Section \ref{section:gradbnd}) which controls the size of the coupling terms between the evolution of $\phi$ and $\psi$.  In this case, the coupling is strong and the gradient bound is not sufficient.

Our motivation is to better understand Example \ref{example:scarring}: in the case $\psi_{init} = c_0 \phi_{init}$ exactly, there is a forward evolutions which continue to have infinite $|\Rm|$, and both $\phi$ and $\psi$ continue to be $0$ at $s=0$.  It is conceivable that there are also forward evolutions in which either $\phi$ or $\psi$ expands. As mentioned, when $\psi_{init} = c_0\phi_{init} \left(1 + |\log(\phi_{init})|^{-1} \right)$ the $\psi$ factor expands and the flow heals with an Ivey soliton at the tip.  What if we consider $\psi_{init} = c_0 \phi_{init} \left( 1 + \epsilon |\log(\phi_{init})|^{-1} \right)$ and send $\epsilon \searrow 0$?  For each $\epsilon$ the initial metric has a smooth forward evolution, with an Ivey soliton in the tip region.  Unfortunately $\epsilon$ controls the difference in the size of the two sphere factors at the edge of the tip region.  As $\epsilon \searrow 0$ the metric in the tip region approaches the singular Bryant soliton where the $S^{p+q}$ is replaced by the Einstein $S^p \times S^q$, as in the singular forward flow in Example \ref{example:scarring}.  Therefore we find it unlikely that a smooth forward evolution exists.

\subsection{Description of the Proof}

Work on this project honestly began with the formal asymptotic calculations in Section \ref{formal_section}.  These calculations gave an initial idea of how the metric may evolve, and importantly give the basis for constructing the barriers in Section \ref{barriers_section}.  Despite this work of constructing barriers being at the core of the proof, we place it later.  Sections \ref{proof_main} is lighter reading and may be understood without knowing all the details in Section \ref{barriers_section}.

Section \ref{proof_main} (and in particular Lemma \ref{convergence_lemma}) proves the Theorem \ref{main_theorem}, using the barriers constructed in Section \ref{barriers_section}.  The main theorem is proved by constructing mollified flows which have uniform existence time and approximate the initial data.  In \cite{ACK} uniform existence time was found by using a bound for the difference of two sectional curvatures, which implies a bound on the difference between first and second derivatives of the warping function.  We were not able to bring that bound to our situation.  Instead we use some regularity results for parabolic PDE to control solutions to Ricci flow between our barriers.

Another difference here is that we do not use collars as in Section 6.1 of \cite{ACK}.  The collars are used to show that the solution to a PDE does not cross some barriers at the boundary of where they are defined.  The collars work because they are themselves barriers, but it is easy to verify that the solution does not cross the collars at the boundary of where the collars are defined.  These would have been technically more difficult to construct in our situation.  A simple alternative would be to use the pseudolocality theorem.  In our situation we can also use regularity to get a uniform speed limit on the solution at the boundary of where our barriers are defined.  One has to do some gymnastics to make sure that the boundary of the barrier region is in the interior of a domain where the solution satisfies a uniformly parabolic equation.  This trick is applicable to other problems where one needs to glue parabolic equations.

\section{Warped Products}
\subsection{Ricci flow on singly-warped products}

In this section review the case where a solution to Ricci flow on $(B^{n_B} \times F^{n_F})$ is given by 
\begin{align}
  \label{singly_warped_form}
  g(t) = g_B(t) + \phi(t) g_F
\end{align}
with $\phi(t): B \to \real$ for each $t$.  The ansantz \eqref{singly_warped_form} is only preserved by Ricci flow if $(F^{n_F}, g_F)$ an Einstein metric: $\Rc(g_F) = \mu g_F$.  In that case the Ricci curvature is given as follows:
\begin{align}
  \Rc(X,X) &= \Rc_{g_B}(X,X) - n_F \nabla_X \nabla_X \lphi - n_F(\nabla_X\lphi)^2 \label{rc_base}\\
  \Rc(V,V) &= (\mu - \phi^2(\lap_{g_B} \lphi + n_F |\nabla \lphi|^2)) g_F(V,V) \label{rc_fiber}\\
  \Rc(X,V) &= 0 \label{rc_cross}
\end{align}
Here $X$ is tangent to the base $B$ and $V$ is tangent to the fiber $F$.

Under Ricci flow, the metric $g_B$ on $B$ and the function $\phi$ evolve by
\begin{align}
  \partial_t g_B &= - 2 \Rc[g_B] + 2 n_F (\nabla \nabla)^B \lphi + 2 n_F \nabla \lphi \otimes \nabla \lphi , \label{gB_evo_single}\\
  \partial_t \phi &= \phi \lap_{g}\lphi - \mu \phi^{-1}
                    = \phi\lap_{g_B} \lphi + n_F\phi|\nabla \lphi|^2 - \mu \phi^{-1}  \\
                  &= \lap_{g_B}\phi - (n_F-1) \frac{(\mu/(n_F-1)) - |\nabla \phi|^2}{\phi^2}\label{phi_evo_single}.
\end{align}
We find the second equation to be pleasing because Ricci flow shows its reaction-diffusion nature quite explicitly.

The system \eqref{gB_evo_single}, \eqref{phi_evo_single} does not depend on $(F, g_F)$ but only on $n_F$ and the constant $\mu$ such that $\Rc_F = \mu g_F$.  Consider the case when $(F, g_F)$ is a space of constant sectional curvature $(\mu/(n_F-1))$.  Then the sectional curvature of a plane tangent to the $F$ factor is given by
\begin{align}
  \label{eq:59}
  L = \frac{(\mu/(n_F-1)) - |\nabla \phi|^2}{\phi^2}
\end{align}
and the sectional curvature of a plane spanned by a vector on the $F$ factor and a unit vector $X$ on the $B$ factor is given by
\begin{align}
  \label{eq:70}
  -(\nabla_X \lphi)^2 - \nabla_X\nabla_X\lphi,
\end{align}
and the sectional of planes tangent to the base are the same as the sectional curvatures for $g_B$.  Usually we will have $\mu/(n_F-1) = 1$ and $F = S^{n_F}$.

\subsection{Doubly warped products}
A metric of the form \eqref{dwp_form} is a singly warped product in two ways; we can consider the fiber to be the $S^p$ factor of the $S^q$ factor.  Such a metric evolves by Ricci flow if
\begin{align}
  \partial_t \fixx \psi &= \partial_s^2\psi
                          + \left( p \frac{\partial_s\psi}{\psi} + q \frac{\partial_s\phi}{\phi} \right)\partial_s\psi
                          - \psi^{-1}\left(\partial_s\psi\right)^2 - (p-1)\psi^{-1}, \label{evo_psi}\\
  \partial_t \fixx \phi &= \partial_s\phi
                          + \left( p \frac{\partial_s\psi}{\psi} + q \frac{\partial_s\phi}{\phi} \right)\partial_s\phi
                          - \phi^{-1}\left(\partial_s\phi\right)^2 - (q-1)\phi^{-1}, \label{evo_phi}\\
  \partial_t \fixx \log \sprime &= p\psi^{-1}\partial_s^2\psi + q\phi^{-1}\partial_s^2\phi. \label{evo_sprime}
\end{align}
Here $\partial_s$ is the vector field $\frac{1}{s'}\partial_x$.  We use the notation $\partial_t \fixx$ to mean a time derivative taken with the $x$ coordinate fixed; space derivatives are always taken with the time coordinate fixed.  The evolution can be derived from \eqref{phi_evo_single} with the knowledge that
\begin{equation*}
  \partial_s^2 + \left( q \frac{\psi_s}{\psi} + p \frac{\phi_s}{\phi} \right) \partial_s
\end{equation*}
is the laplacian for $g$, acting on a function which is constant on the $S^p, S^q$ factors.  (Compare this to the laplacian acting on radially symmetric functions in $(\real^{p+1}, g_{std} = ds^2 + s^2 g_{S^p})$: $(\partial_s^2 + p s^{-1}\partial_s)$.)

These metrics have five sectional curvatures:
\begin{align}
  \label{eq:101}
  L_\phi = \frac{1- \left( \partial_s \phi \right)^2 }{\phi^2}, \quad L_\psi = \frac{1 - \left(\partial_s \psi \right) }{\psi^2}
\end{align}
are the sectional curvatures tangent to the $S^q$ and $S^p$;
\begin{align}
  \label{eq:102}
  K_\phi = -\frac{\partial_s^2 \phi }{\phi}, \quad K_\psi = -\frac{\partial_s^2 \psi}{\psi}
\end{align}
are the sectional curvatures for a plane with tangents $\partial_s$ and a vector on one of the spheres; and
\begin{align}
  \label{eq:103}
  J = \frac{\phi_s \psi_s}{\phi \psi}
\end{align}
is the sectional curvature for planes spanned by a pair of vectors tangent to each of the spheres.

\subsection{Boundary conditions}
The topology and smoothness of the metric-space completion of
\begin{align}
  s'(x)^2 dx^2 + \psi(x)^2 g_{S^p} + \phi(x)^2 g_{S^q}
\end{align}
depends on the boundary conditions at $x = 0$ and $x=1$.  Let us concentrate on the left end, the $x = 0$ boundary.  First, if $s'(x)$ is not integrable near $x=0$, then the left end has infinite length and the completion has the topology $(-\infty, 0) \times S^p \times S^q$ near there.  

Now, consider the case when $s'(x)$ is integrable near $x=0$.  If
\begin{align}
  \label{eq:80}
  \psi(0) = 0, \quad \phi(0) = 0,
\end{align}
then the left end closes to the cone over $S^p \times S^q$ and cannot be a smooth manifold; this is the case with our initial data.  If
\begin{align}
  \psi(0) > 0,\quad \phi(0) = 0, \label{warpingfn_bcs}
\end{align}
then the left end closes to $D^{q+1} \times S^p$; this is the case with our forward evolution.  More conditions are needed for the metric to be smooth, namely that
\begin{align}
  \partial_s \psi \bigg|_{x = 0} = 0, \quad \partial_s \phi \bigg|_{x = 0} = 1, \label{derivative_bcs}
\end{align}
and furthermore any odd derivatives of $\psi$ vanish, and any even derivatives of $\phi$ vanish.

\subsection{The gradient bound}\label{section:gradbnd}

There is an important gradient bound for Ricci flow on a singly-warped product $g_B + \phi^2 g_F$.  From the singly-warped evolution equation \eqref{gB_evo_single}, \eqref{phi_evo_single} we can use the Bochner formula to derive that the \emph{scale invariant} quantity $b = |\nabla \phi|^2$ satisfies
\begin{align}
  \label{eq:54}
  \partial_t b
  &= \lap_{g} b  - 2 |\nabla \nabla \phi|_{g_B}^2 - 2 g(\nabla b, \nabla \lphi ) \\
  &+ 2 b \phi^{-2}(\mu - (n_F-1)b).
\end{align}
In our case $\mu > 0$.  Applying the maximum principle in case easily yields the following lemma.
\begin{lemma}\label{singly_gradbnd}
  Suppose $(B, g_B(0))$ is complete and compact.  If a smooth solution $g(t)$ to Ricci flow is of the form \eqref{singly_warped_form} and $\mu > 0$, then a global bound on $|\nabla (e^u)|^2$ by $\frac{C}{n_F-1}$ is preserved for any $C \geq \mu$. 
\end{lemma}
Taking $C = \mu$ above says that the sectional curvature $L = \phi^{-2}(\frac{\mu}{n_F-1} - |\nabla \phi|^2)$ has non-negativity preserved.

In our doubly-warped product situation, the metric on the base is not compact but we have the boundary conditions \eqref{derivative_bcs}.  Therefore, remembering that we take $\mu = n_F - 1$ we have the following.
\begin{lemma}\label{doubly_gradbnd}
  Ricci flow on a doubly warped product of the form \eqref{dwp_form} with boundary conditions \eqref{warpingfn_bcs}, \eqref{derivative_bcs} preserves the bounds
  \begin{align}
    \label{eq:95}
    |\partial_s \psi|^2 \leq C, \quad |\partial_s \phi|^2 \leq C
  \end{align}
  for any $C \geq 1$.
\end{lemma}

In the case $\mu \leq 0$ the evolution equation for $b$ \eqref{eq:54} is quite strong, because the reaction terms are nonpositive.  Lott and Sesum \cite{lottSesum} exploit this and more to get a very complete characterization of Ricci flow for warped products of $S^1$ over surfaces.
Here is another consequence of this equation when $\mu \leq 0$, which we use in Section \ref{formal_outer} to identify a steady soliton.
\begin{lemma}\label{eternal_constant}
  If an ancient solution to Ricci flow is of the form \eqref{singly_warped_form} with $\mu \leq 0$, and $|\nabla \log \phi|^2$ is asymptotically zero at infinity, then $\phi$ is in fact constant.  
\end{lemma}
\begin{proof}
  This instead uses the evolution of $\tilde b = |\nabla \log \phi|^2$:
  \begin{align}
    (\partial_t - \lap_g) \tilde b
    &= - 2 |\nabla \nabla \log \phi|^2 + 4 \mu \phi^{-2} \tilde b - 2 n \tilde b^2 \\
    &\leq - 2n \tilde b^2
  \end{align}
  By applying the maximum principle at arbitrarily negative times, we show that $\tilde b$ is arbitrarily small.
\end{proof}

\subsection{Choosing coordinates}
In the evolution \eqref{evo_psi}, \eqref{evo_phi}, \eqref{evo_sprime} we see quite explicitly the degenerate parabolic nature of Ricci flow, coming from diffeomorphism invariance of Ricci flow.  We are looking at the evolution of a system $(s', \phi, \psi)$ but the second derivative of $s'$ makes no appearance.  Diffeomorphisms of the manifold which respect the doubly warped product structure are given by diffeomorphisms
\begin{align}
  (0, 1) \times S^p &\times S^q \to (\alpha_1,\alpha_2) \times S^p \times S^q, \\
  (x, a, b) &\mapsto (f(x), a, b).
\end{align}

We can use diffeomorphism invariance to reduce the number of functions in our system, by choosing a representative of the diffeomorphism class. This is a fancy way of saying that we want to choose a more geometric coordinate for the interval $(0, 1)$.  A natural choice is to choose a point $x_0$ and let
\begin{align}
  \label{eq:12}
  s(x) = \int_{x_0}^x s'(\tilde x) d\tilde x
\end{align}
which represents the signed distance, with respect to $g$, from $x=x_0$.  A disappointing aspect of this coordinate is that
\begin{align}
  \label{eq:22}
  \partial_t \fixx s(x)
  = \int_0^x \partial_t \fixx s'd\tilde x
  = \int_0^{s(x)} \partial_t \fixx \log( s') d \tilde s,
\end{align}
so the evolutions $\partial_t \fixs \psi$ and $\partial_t \fixs \phi$ are non-local.

In Figure \ref{figure:phipsi} we illustrate the flow in the $s$ coordinate.

\begin{figure}
  \centering
  \centerline{\includegraphics[scale=.4]{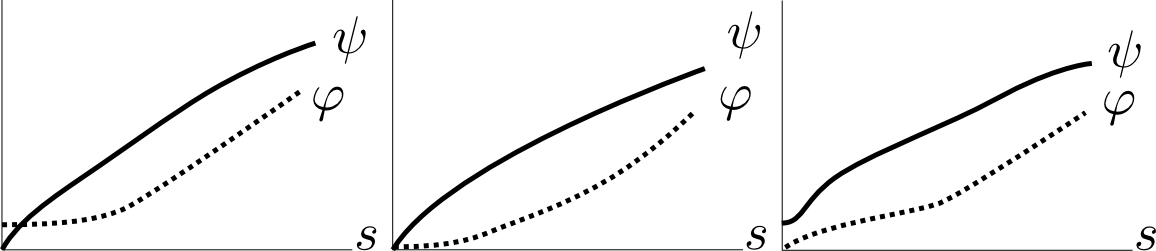}}
  \caption{
    An illustration $\phi$ and $\psi$ as functions of $s$, before, at, and after the singular time.  We don't use exact numerics but the boundary conditions are accurate.
  }
  \label{figure:phipsi}
\end{figure}

\subsection{The $r$ coordinate}\label{r_coordinate_section}
In a region where $\phi$ is increasing, we may use it as a coordinate.  We set $r = \phi$ to emphasize when it is being used as a coordinate.  Using $r$, the metric can be written
\begin{align*}
  g = \frac{1}{v}dr^2 + h g_{S^p} + r^2 g_{S^q}
\end{align*}
where $v = (\phi_s)^2$ and $h = \psi^2$.  
The evolution in these coordinates is given by,
\begin{align}
  \partial_t v
  &= \fopv(v,h) \label{fopv_def}\\
  &= v v_{rr} - \oh v_r^2 + r^{-1}(q - 1 - v)v_r \\
  &+ 2 (q-1) r^{-2}v(1-v)  - 2  p\frac{v}{h} \left( \frac{vh_r^2}{4h^2} \right) \\
  \nonumber\\
  \partial_t h
  &= \foph(v,h) \label{foph_def}\\
  &= v h_{rr} + r^{-1}(q-1+v)h_r \\
  &- 4 \left( \frac{vh_r^2}{4h} \right) - \alpha_p^2 \nonumber\\
\end{align}
Of course in these coordinates the gradient bound for $\phi$ from Lemma \ref{doubly_gradbnd} is just
\begin{align}
  v \leq 1. \label{v_gradbnd}
\end{align}
Note that
\begin{align}
  \label{eq:96}
  u \defeq \frac{vh_r^2}{4h} = (\psi_s)^2,
\end{align}
so the gradient bound for $\psi$ from Lemma \ref{doubly_gradbnd} takes the form
\begin{align}
  \label{h_gradbnd}
  \frac{v h_r^2}{4h} \leq 1.
\end{align}
This bound will be sufficient for dealing with that term.  Let us define,
\begin{align}
  \gopv(v, h, u)
  &= v v_{rr} - \oh v_r^2 + r^{-1}(q-1-v)v_r \label{gopv_def}\\
  &+ \alpha_q^2 r^{-2}v(1-v) - 2 pu \frac{v}{h}, \\
  \goph(v, h, u)
  &= v h_{rr} + r^{-1}(q-1+v)h_r \label{goph_def}\\
  &- 4 u - \alpha_p^2 .
\end{align}
Then we may write
\begin{align}
  \label{evo_gops}
  \partial_t v = \gopv\left( v, h, \frac{v h_r^2}{4h} \right), \quad \partial_t h = \goph\left( v, h, \frac{v h_r^2}{4h} \right).
\end{align}

\begin{figure}[tp]
  \centering
  \centerline{\includegraphics[scale=.4]{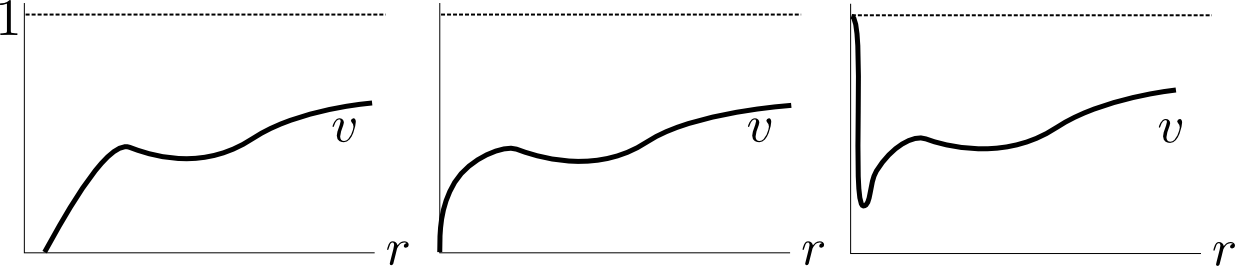}}
  \caption{
    An illustration of the function $v$ in terms of $r$ before, at, and after the singular time.  We do not use exact numerics but do pay attention to asymptotics at $r=0$.
  }
\end{figure}

\section{Existence}\label{proof_main}
In this section we prove the existence theorem, Theorem \ref{main_theorem}, except for the construction of barriers which is done in Section \ref{barriers_section}.  Sections \ref{max_princ_section} and \ref{regularity_section} prove things about the evolution near the singular point.  These Sections are purely on the level of looking at a parabolic PDE on an interval.  In Sections \ref{mollified_defns}, \ref{uniform_crossing_time}, and \ref{convergence} we construct a forward evolution as a limit of forward evolutions from smooth initial metrics.  A small amount of work is need to glue the results about regularity near the singular point to regularity results for Ricci flow in the smooth region.

\subsection{Avoidance Principle and Barriers}\label{max_princ_section}
% Recall the evolution equation for $v$ and $h$:
% \begin{align}
%   \dopv (v,h) &\defeq \left(\partial_t \fixr - \fop^{(v)}(v,h) \right) v = 0\label{evo_dopv}\\
%   \doph (v,h) &\defeq \left(\partial_t \fixr - \fop^{(h)}(v,h) \right) h = 0\label{evo_doph}
% \end{align}
% Where:
% \begin{align}
%   \label{eq:13}
%   \fopv(v,h)
%   &= v v_{rr} - \oh v_r^2 - r^{-1}vv_r \\
%   & + \oh \alpha_q^2 r^{-2}v_r + \alpha_q^2 r^{-2} v(1-v)\\
%   & - \oh p \frac{v}{h} \left(v \frac{h_r^2}{h} \right)\\
%   \foph(v,h)
%   &= v \left( h_{rr}  + r^{-1} h_r \right) - \left( v \frac{h_r^2}{h} \right)\\
%   &+ \oh \alpha_q^2 r^{-1}h_r - (p-1) \\
% \end{align}
We use the avoidance principle to trap solutions ($v, h$) to \eqref{fopv_def}, \eqref{foph_def} in a simple ``square'':
\begin{align}
  \label{eq:9}
  \{(v, h):\;
  &\text{ for all } (r,t) \in [r_1, r_2] \times [t_1, t_2] \} \\
  &(v(r,t), h(r,t)) \in [v^-(r, t), v^+(r,t)] \times [h^-(r,t), h^+(r,t)] 
\end{align}
In general, the term in the evolution of $v$ which involves derivatives of $h$,
\begin{align}
  \label{eq:31}
  -2p \frac{v}{h} \psi_s^2 = - \oh v^2\frac{h_r^2}{h}  
\end{align}
could cause difficulty in using an avoidance principle; if $(v,h)$ is on the $v = v^+$ boundary of such a set we cannot have any knowledge about $h_r$.  In our situation, we are saved by the fact that we have the preserved gradient bound \eqref{h_gradbnd}.  Luckily, this term never dominates our calculations, so we never need fine control on it and the following ad hoc definition of barriers suffices.
\begin{definition}[Barriers]\label{barrier_def}
  Functions $v^{+}$, $v^{-}$, $h^{+}$, $h^{-}$ which are $C^2$ on $[r_1, r_2] \times [t_1, t_2]$ are called barriers if
  \begin{itemize}
  \item \textbf{The barriers are properly ordered:}
    \begin{align}
      \label{eq:30}
      v^- < v^+, \quad h^- < h^+.
    \end{align}
  \item \textbf{$v^{\pm}$ are sub/supersolutions:}
    For all $h^*$ with $h^{-} \leq h^*$ and for all $u^*$ with $0 \leq u^* \leq 1$ we have
    \begin{align}
      \partial_t v^- - \gopv(v^-, h^*, u^*) \leq 0 \leq \partial_t v^+ - \gopv(v^{+}, h, u^*). \label{subsuper_v_ineq}
    \end{align}
  \item \textbf{$h^{\pm}$ are sub/supersolutions:}
    For all $v^*$ with $v^{-} \leq v^* \leq v^{+}$ and for all $u^*$ with $0 \leq u^* \leq 1$ we have
    \begin{align}
      \partial_t h^- - \goph(v^*, h^-, u^*) \leq 0 \leq \partial_t h^+ - \goph(v^*, h^+, u^*). \label{subsuper_h_ineq}
    \end{align}
  \item \textbf{Smoothness conditions at zero:} 
    If $r_1 = 0$ then
    \begin{align}
      \label{eq:75}
      v^-(0, t) = v^+ (0, t) = 1,\quad \partial_r v^- (0, t) = \partial_r v^+(0,t) = 0, \\
      0 < h^-(0, t) < h^+(0, t), \quad \partial_r h^-(0, t) = \partial_r h^+(0, t) = 0.
    \end{align}
  \end{itemize}
\end{definition}

Our equations will have Neumann boundary conditions at $r=0$, and we may think of $r=0$ as an interior point.  Therefore we use the following definition of parabolic boundary.
\begin{definition}
  \label{eq:51}
  The parabolic boundary of  $[r_1, r_2] \times [t_1, t_2]$ is defined to be
  \begin{align}
    \label{eq:76}
    \left( [r_1, r_2] \times \{t_1\} \right)
    \cup
    \left( \{r_1\}\times [t_1, t_2] \right)
    \cup
    \left( \{r_2\}\times[t_1, t_2]  \right)
  \end{align}
  if $r_1 > 0$ or
  \begin{align}
    \label{eq:77}
    \left( [0, r_2] \times \{t_1\} \right)
    \cup
    \left( \{r_2\} \times [t_1, t_2] \right)
  \end{align}
  if $r_1 = 0$.
\end{definition}

For our barriers, we have the following avoidance principle.  We state it with a slight contrapositive compared to the usual statement; we say ``if a solution crosses the barriers, it crosses on the parabolic boundary first''.
\begin{lemma}\label{avoidance_princ}
  Suppose that $v^{+}$, $v^{-}$, $h^{+}$, $h^{-}$ are barriers on $[r_1, r_2] \times [t_1, t_2]$, $v$ and $h$ evolve by \eqref{evo_gops} and $h$ satisfies the gradient bound \eqref{h_gradbnd}.  Define $L^- = r^{-2}(1-v^-)$, $L = r^{-2}(1-v)$, $L^+ = r^{-2}(1-v^+)$.  If $r_1 = 0$ then suppose these functions are defined at $r=0$ by their limit as $r \searrow 0$ and that $L^-, L, L^+, h^-, h, h^+$ satisfy the Neumann condition $\partial_r (\cdot) = 0$ at $r=0$.

  Let $t_0$ be the infimum of the times for which one of the inequalities
  \begin{align}
    -L^{-} < -L < -L^+, \quad h^- < h < h^+, \label{subsuperbnd3}.
  \end{align}
  is violated.  Then there is an $r_0$ so that one of the inequalities \eqref{subsuperbnd3} is violated at $(r_0, t_0)$ and $(r_0, t_0)$ is on the parabolic boundary of $[r_1, t_1]$.
\end{lemma}
\begin{proof}
  The case $r_1 > 0$ is a standard basic comparison principle argument. 

  To deal with the case $r_1 = 0$, note that using the Neumann condition, the equation for both $L$ and $h$ allows for comparison at $r= 0$.  (Essentially this is because $r = 0$ is secretly an interior point of the symmetric manifold.)
\end{proof}

We will apply the above lemma where $r_1$ and $r_2$ depend on $t$, or $r_1 \equiv 0$.  The proof is the same.  We will also need to glue barriers which are defined in different space-time regions.
\begin{definition}[Gluing Conditions]
  Suppose $(\hat v^-, \hat v^+, \hat h^-, \hat h^+)$ are barriers are $[\hat r_1, \hat r_2] \times [t_1, t_2]$ and $(\tilde v^-, \tilde v^+, \tilde h^-, \tilde h^+)$ are barriers on $[\tilde r_1, \tilde r_2] \times [t_1, t_2]$.  We say these satisfy the gluing conditions if     $\hat r_1 < \tilde r_1 < \hat r_2 < \tilde r_2$ and 
  \begin{align}
    \label{eq:2}
    \hat v^-(\tilde r_1, t) > \tilde v^-(\tilde r_1, t), \quad \tilde v^-(r_2, t) > \hat v^-(r_2, t), \\
    \hat v^+(\tilde r_1, t) < \tilde v^+(\tilde r_1, t), \quad \tilde v^+(r_2, t) < \hat v^+(r_2, t). 
  \end{align}
  And similarly for $\hat h^-, \hat h^+, \tilde h^-, \tilde h^+$.
\end{definition}
If two sets of barriers satisfy the gluing conditions, then we can conclude that a solution is trapped between both on $[\hat r_1, \tilde r_2] \times [t_1, t_2]$ if we only check the boundary conditions on the parabolic boundary of $[\hat r_1, \tilde r_2] \times [t_1, t_2]$.  One way to see this is to go to the proof of Lemma \ref{avoidance_princ} and check that a solution cannot touch $\max(\hat v^-, \tilde v^-)$ from above at a point where it transitions from $v^-$ to $\tilde v^-$.  Perhaps a cleaner way is to just use Lemma \ref{avoidance_princ} twice.  Applied to the hat barriers it says a solution has to have its first crossing of the hat barriers on the parabolic boundary of $[\hat r_1, \hat r_2] \times [t_1, t_2]$.  But if it crosses at $r = \hat r_2$ first, then by the gluing conditions it has crossed the tilde barriers at $\hat r_2 \in [\tilde r_1, \tilde r_2]$ before $\tilde r_1$, so by Lemma \ref{avoidance_princ} it must have crossed the tilde barriers at $\tilde r_2$ first.

\begin{table}
  \centering
  \def\arraystretch{1.5}
  \begin{tabular}{c | c | c}
    Region & Barrier for $v$ & Barrier for $h$\\
    \hline
    Outer  & $a^2_\pm(1 + (1 \pm \epsilon)\alpha_q^2 t r^{-2} )v_0$
                             & $b^2_{\pm}(1 + (1 \pm \epsilon)\alpha_q^2 t r^{-2} )h_0$ \\
    \hline
    Parabolic & $2 \hat a^2_{\pm} (1 + \alpha_q^2 (1 \pm D |\tau|^{-1}\rho^{-2})\rho^{-2} )$
                             & $\oh \hat b^2_{\pm} (\rho^2 + \alpha_q^2 (1 \pm D |\tau|^{-1}\rho^{-2}))$ \\
    \hline
    Tip & $\bry(c_{\pm} \zeta) + (1 \mp \epsilon) |\tau|^{-1}c^{-2}_\pm \cry(c_{\pm} \zeta) $
                             & $ \alpha_q^4 \bar c_{\pm}^2 + |\tau|^{-1} \ary(c_{\pm}\zeta)$
  \end{tabular}
  \caption{The barriers used in each region described in Lemma \ref{barriers_full}}
  \label{barrier_table}
\end{table}

The construction of barriers is conceptually simple, but it introduces a number of unsightly constants and has been relegated to Section \ref{barriers_section}.  The reader who wishes to know where the asymptotics of the forward evolution come from may read the formal asymptotics in Section \ref{formal_section}.  The result of the lemmas of Section \ref{barriers_section} is the following.
\begin{lemma}\label{barriers_full}
  Let $\delta, \epsilon > 0$ be given and uniformly small.  Then there is $r_*>0 $ and $t_* > 0$ so that $[0, r_*] \times (0, T_*]$, is covered by three sets of barriers.
  Specifically, there are constants $C_1, C_2$ independent of $\epsilon, \delta$ so that
  \begin{itemize}
  \item Let $\rho = r/\sqrt{t}$.  In the outer region given by,
    $$
    \{\rho \geq C_1/\sqrt{\epsilon} \text{ and } r < r_*\}
    \eqdef
    \{(r, t): \sqrt{\frac{t}{\epsilon}}C_1 < r < r_* \},$$
    the barriers are the outer region barriers in Table \ref{barrier_table}, detailed in Lemma \ref{outer_barriers}.
  \item Let $\zeta = |\log t| \rho$.  In the parabolic region given by,
    $$
    \{\rho < 2 C_1/\sqrt{\epsilon} \text{ and } \zeta > C_2/\sqrt{\epsilon} \}
    \eqdef
    \left\{ (r, t) : \sqrt{\frac{t}{\epsilon|\log t|}}C_2 < r < \sqrt{\frac{t}{\epsilon}}2C_1  \right\},
    $$
    the barriers are the parabolic region barriers in Table \ref{barrier_table}, detailed in Lemma \ref{para_barriers}, with constants
    \begin{align}
      \label{eq:56}
      \hat a^2_{\pm} = k^2 + o\left(|\epsilon| + |\delta|\right), \quad \hat b^2_{\pm} = k^{-2} + o\left(|\epsilon| + |\delta|\right)
    \end{align}
  \item In the tip region given by,
    $$
    \{\zeta < 2C_2/\sqrt{\epsilon} \}
    \eqdef
    \left\{ (r, t) : r < \sqrt{\frac{t}{\epsilon |\log t|}}2C_2 \right\}
    $$
    the barriers are the tip region barriers in Table \ref{barrier_table}, detailed in Lemma \ref{tip_barriers}, with constants
    \begin{align}
      \label{eq:63}
      \bar c_\pm = c_0 + o(|\epsilon| + |\delta|) , \quad c_\pm = c_0 + o(|\epsilon| + |\delta|) 
    \end{align}
  \item The gluing conditions are satisfied between the outer and parabolic region barriers, and between the parabolic and tip region barriers.
  \end{itemize}
\end{lemma}
By design the outer barriers contain the initial data.  Therefore by the first point, any metric with the asymptotics \eqref{assumed_asymptotics} will be contained by the barriers at time zero, in a smaller region $[0, r_*]$.  Also, any functions between the barriers for positive time will satisfy the assumptions for Lemma \ref{regularity_tip} below.  This is because the barriers in the tip region trap $v$ in the correct manner near $r=0$.  We will not (and cannot) directly apply Lemma \ref{avoidance_princ} to the barriers on $[0, T_*]$, but we will apply it on $[t_1, T_*]$.

\subsection{Parabolic Regularity}\label{regularity_section}

In this section we apply regularity results for parabolic partial differential equations to our situation.  We prove two lemmas which will allow us to get regularity for solutions provided $v$ and $h$ stay in bounded regions.

\begin{lemma}\label{regularity_tip}
  Let $v\geq 0$ and $h\geq 0$ satisfy \eqref{fopv_def}, \eqref{foph_def} on $[0, r_2] \times [t_1, t_2)$, with
  \begin{align}
    \label{eq:66}
    \partial_r v(0, t) = \partial_r h (0, t) = 0.
  \end{align}
  Suppose the gradient bounds \eqref{v_gradbnd} and \eqref{h_gradbnd} are satisfied.  Let $f = \frac{1-\sqrt{v}}{r^2} = \frac{1-\phi_s}{\phi^2}$. Suppose
  \begin{align}
    \sup_{[0, r_2] \times [t_1, t_2)} f < C_1, \qquad \sup_{[0, r_2] \times [t_1, t_2)} h < C_1, \label{f_bound}
  \end{align}
  \begin{align}
    \inf_{[0, r_2] \times [t_1, t_2)} v > C_1^{-1}, \qquad \inf_{[0, r_2] \times [t_1, t_2)}  h > C_1^{-1},
  \end{align}
  and let $t_1' > t_1$ and $r_2' < r_2$ be given.

  Then there is $C_2 < \infty$ depending on $C, r_2, r_2', t_1,t_1', t_2, p, q$ such that
  \begin{align}
    \label{eq:6}
    \sup_{[0, r_2'] \times [t_1', t_2)} |\Rm| < C_2
  \end{align}
  where $\Rm$ is the curvature of the associated doubly warped product metric.
\end{lemma}
\begin{proof}
  The sectional curvatures are smooth functions of $h$, $h_r$, $h_{rr}$, and $f$, $f_r$, so it suffices to control those.

  The function $f$ satisfies
  \begin{align}
    f_t
    &= \left[ 1-2fr^2 + f^2 r^4 \right] f_{rr}  \\
    &+ \frac{q+2}{r} f_r + \left[3r^3f^2 - 6rf \right]f_r \\
    &+ \left[-r^2 q f + r^2 f + 3q -3 - \frac{p}{2} \frac{1}{h} \frac{vh_r^2}{h} \right] f^2
    \label{evo_f}
  \end{align}
  Think of $f$ as a rotationally symmetric function in $B(0, r_2) \subset \real^{q+3}$.  Then by our assumptions the equation \eqref{evo_f} is a uniformly parabolic equation with coefficients which depend on space and time but have an $L^{\infty}$ bound.  By the De Giorgi-Nash-Moser theorem (Theorem 6.28 in \cite{liebz} works here) $f$ is in the H\"older space $C^\alpha(B(0, r_2') \times [t_1', t_2))$ with $\alpha$ and the $C^\alpha$ norm only depending on the allowed numbers.

  We know already that $h$ is uniformly in $C^{1}$, because we assumed the gradient bound $h_r^2 \leq \frac{h}{v}$.  Now that we know $v$ is uniformly H\"older, the coefficients in the evolution for $h$ are uniformly H\"older.  We can multiply by a cut-off function to get a solution to an inhomogeneous equation with zero boundary data.  Then the Schauder estimate gives a uniform bound on $|h_{rr}|$.  Theorem 3 of \cite{schauder_schlag} has a clean statement of the Schauder estimate in a more general case than we need here.  (At the moment we are not actually applying the Schauder estimate to a system, only to the equation for $h$.)

Once we have the uniform bound on $|h_{rr}|$ we get that the coefficients in the evolution of $f$ are H\"older, so applying the Schauder estimate gives a $C^2$ bound on $f$.
\end{proof}

Outside of the region where $r=0$, the relevant functions satisfy nice parabolic equations in one dimension, and we can actually get regularity up to the starting time $t=0$.
\begin{lemma}\label{regularity_outer}
  Let $v$ and $h$ be as in the hypotheses of Lemma \ref{regularity_tip}, but instead defined on $[r_1, r_2] \times [0, t_2)$ with $r_1 > 0$.  Suppose also that at $t=0$ both have $H^1([r_1, r_2])$ bounds:
  \begin{align}
    \label{eq:1}
    \int_{r_1}^{r_2} (\partial_r v(r, 0))^2  dr < C_1, \qquad     \int_{r_1}^{r_2} (\partial_r h(r, 0))^2  dr < C_1.
  \end{align}
  
  Let $r_1'$ and $r_2'$ be given with $r_1 < r_1' < r_2' < r_2$.  There is a $C_2$ depending on $C_1, r_1, r_1', r_2, r_2', p, q$ so that
  \begin{align}
    \label{eq:10}
    \sup_{[r_1', r_2'] \times [0, t_2)} |\Rm| < C_2
  \end{align}
\end{lemma}
\begin{proof}
  Taking $\eta$ to be a static smooth cut-off function, $v\cdot \eta$ satisfies a parabolic equation on $[r_1, r_2]$ with zero boundary data, bounded coefficients and bounded right hand side.  (In particular, the terms involving $1/r$ are bounded because $r_1 > 0$).  Therefore for all $t$ the solution is uniformly in $H^1([r_1, r_2])$  (Theorem 6.1 of \cite{liebz}).  The Sobolev embedding in one dimension tells us that $v$ is in $C^\alpha$ for $\alpha \leq 1/2$.  Then as in the proof of Lemma \ref{regularity_tip}, the conclusion follows from the Schauder estimate.
\end{proof}

\subsection{Mollified flows}\label{mollified_defns}
We create smooth initial metrics which approximate our singular metric $g_{init}$.

\begin{definition}
  Let $g_{init}$ be given satisfying the assumptions of Theorem \ref{main_theorem}.  Let $\{g_{\omega}\}_{\omega  \in (0, \omega_*)}$ be a one parameter family of smooth metrics on $M$ such that
  \begin{itemize}
  \item $\rest_{r > \sqrt{\omega}}$ is identical for $g_{init}$ and $g_{\omega}$, and the metrics agree there.
  \item $g_\omega$ is a smooth metric and satisfies the gradient bound \eqref{grad_bnd}.
  \item For some fixed $r_\#$, $\phi_\omega$ is increasing for $r < r_\#$, so the $r$ coordinate is defined on $\north_{r<r_\#}$.  (Recall Definition \ref{north} for the notation $\north_{r < r_\#}$.)  In terms of these coordinates, $h_\omega(0)$ and $v_\omega(0)$ lie between barriers from Theorem \ref{barriers_full}, evaluated at a small positive time.  In particular, take the barriers for any small $\epsilon$ and $\delta$, and evaluate them at time
    $$t_{\omega} = \frac{\epsilon}{ \sqrt{C_3}}\omega$$
    for a sufficiently large $C_3$.
  \end{itemize}
\end{definition}
We take $t_\omega \lesssim \omega$ above because we need very little time to pass, so that the barriers still trap the initial function at $r = \sqrt{\omega}$. In particular $(r,t)  = (\sqrt{\omega}, t_\omega)$ is in the outer region.

Each $g_\omega$ has $\sup |\Rm| < C_\omega$.  Shi's short time existence theorem \cite{Shi} gives us a Ricci flow starting from $g_\omega$ on some time interval $[0, T_\omega)$.  (We must turn to \cite{Shi} here because we allow for noncompact initial data.)  We wish to show first that in fact $T_\omega \geq T_0 > 0$.  Then, we want bounds on curvatures away from the singularity.  This allows us to get convergence of a subsequence $g_{\omega_k}$ to a solution to Ricci flow emerging from the initial metric.

  In principle the existence time and regularity works because the relevant functions will satisfy partial differential equations similar to
\begin{align}
  \label{eq:11}
  \partial_t u(x, t) = \partial_x(a(u)\partial_xu) + b(u)\partial_xu + c(u) u
\end{align}
whose solutions are great as long as the coefficients $a(u)$, $b(u)$ and $c(u)$ do not do anything bad.  The barriers will tell us that the coefficients behave reasonably.  There is some extra work because our barriers are not defined on the whole of the manifold $M$, and so the solution might cross the barriers on the edge of the domain where they are defined (where $r = r_*$). This problem is overcome in Lemmas \ref{rest_bound} and \ref{cross_time_bound} because on the complement of the region where the barriers are defined, the initial  metrics are uniformly smooth (and in fact identical).

\subsection{Outer Control}\label{uniform_crossing_time}
Let us define $\Tcross_\omega$ to be the first time that one of the following fails (or $\Tcross_\omega = T_\omega$ if the metric becomes singular before on of the following happen):
\begin{enumerate}
\item\label{item:coord} The function $\phi_\omega$ is not strictly increasing (from the tip) until $\phi_\omega = r_*$, and therefore the $r$ coordinate is defined up to $r=r_*$.
\item\label{item:cross} The functions $v_\omega(r,t)$ and $h_\omega(r,t)$ do not cross the time-shifted barriers at $r = r_*$. 
\end{enumerate}
Thus on the interval $[0, \Tcross_\omega)$ both of the above hold for the metric $g_\omega$.  Both (\ref{item:coord}) and (\ref{item:cross}) both hold for a small time (because they hold for $g_\omega(0)$ and the functions $v_\omega = (\partial_s\phi_\omega)^2$ and $h = \psi_\omega^2$ are continuous in time).  Also (\ref{item:coord}) cannot fail before $v_\omega \searrow 0$ for some $r < r_*$.   Since $v_\omega^{\pm}(r, t) > 0$, the avoidance principle implies that while  (\ref{item:cross}) holds, (\ref{item:coord}) will hold.  So $\Tcross_\omega$ can actually be characterized as the first time when (\ref{item:cross}) fails.

We want to show that $\Tcross_\omega = \min(T_\omega, T_0)$ for $T_0$ independent of $\omega$.  We do this in Lemmas \ref{basic_strip_bound} to \ref{cross_time_bound}.  If $M$ is compact then an alternative way to get the $|\Rm|$ bounds in Lemmas \ref{basic_strip_bound} and \ref{rest_bound} would be to use pseudolocality theorem (Theorem 10.3 of \cite{Perelman}).  As the metrics $g_\omega(0)$ all agree on $\rest_{r > r_*}$ they will satisfy uniform curvature and volume bounds on a small enough scale, so there is a uniform curvature bound on a smaller scale as long as the Ricci flow exists.  This was the method used in \cite{conicalSing}.

There is a technicality in the following discussion.  The set $\north_{r < r_*}$  (defined in Definition \ref{north}) depends on time and $\omega$, but we can think of $v$ and $h$ as defined on the fixed box $[0, r_*] \times [0, \Tcross_\omega]$.  
\begin{lemma}\label{basic_strip_bound}
  Consider the metrics $g_\omega$, which have the $r$ coordinate defined on $\north_{r<r_*}$ and are trapped between the barriers on $[0, r_*] \times [0, \Tcross_{\omega})$.

  Then for some time we have curvature bounds in a smaller region.  Precisely, there is a $T_0$ so that
  \begin{align}
    \label{eq:7}
    \sup_{t \in [0, \min(\Tcross_\omega, T_0)]} \quad \sup_{p \in \north_{\frac{1}{2}r_* < r < \frac{3}{4}r_*} } |\Rm_{g_\omega(t)}|(p) < C_0
  \end{align}
  for some $C_0$ independent of $\omega$.
\end{lemma}
\begin{proof}
  As long as the metric coefficients are trapped between the barriers, they satisfy the hypotheses of Lemma \ref{regularity_outer} on the interval $\left[\frac{1}{4}r_*, r_*\right]$, which uses results for parabolic PDE in one dimension to control $|\Rm|$ on the smaller interval $[\frac{1}{2}r_*, \frac{3}{4}r_*]$.
\end{proof}

The uniform bound on $|\Rm|$ in the strip $\north_{ar_* < r < br_*}$ lets us get a uniform bound on $|\Rm|$ away from $M - \north_{r < ar_*} = \rest_{r > ar_*}$. 

\begin{lemma}\label{rest_bound}
  There is a $T_0$ and $C_0$ independent of $\omega$ such that for $t \in [0, \min(T_0, \Tcross_\omega)]$ and $p \in \rest_{r > ar_*}$ we have
  \begin{align}
    \label{eq:19}
    |\Rm_{g_\omega(t)}|(p) < C_0.
 \end{align}
 Also, there are constants $C_k$ for $k \in \nats$ so that for points $p$ in the smaller region $\rest_{r > br_*}$ we have
 \begin{align}
   \label{eq:67}
   |\nabla^k \Rm_{g_\omega(t)} |(p) < C_k
 \end{align}
\end{lemma}
\begin{proof}
  On $\rest_{r > ar_*}$, $|\Rm|$ is uniformly bounded at the initial time.
  Under Ricci flow $|\Rm|$ satisfies
  \begin{align}
    \label{eq:20}
    \partial_t |\Rm|^2 \leq \lap |\Rm|^2 + c (|\Rm|^2)^{3/2}.
  \end{align}
  Here $c$ depends only on the dimension of the space.  Lemma \ref{basic_strip_bound} tells us that $|\Rm|$ is bounded near the boundary of $\rest_{r > ar_*}$.  The first part of the theorem follows from the maximum principle.

  To get the control on the derivatives, we use the local Shi's estimates given in 14.4.1 of \cite{bookanalytic}.  These control derivatives on a fixed compact set $K$, whereas $\rest_{r \geq r_*}$ is changing in time; this is the technicality mentioned earlier.  This problem is easily overcome because we have bounds on $|\Rm|$ and hence on $\partial_t g_{\omega}$.  This implies that the time-dependent set $\north_{r > br_*}[g_\omega(t)]$ is contained in the fixed set $\north_{r > ar_*}[g_\omega(0)]$ for small time.  We apply the local Shi estimates on this fixed compact set.
\end{proof}

\begin{lemma}\label{cross_time_bound}
  There is a $T_0$ independent of $\omega$ so that $\Tcross_\omega \geq \min(T_0, T_\omega)$.
\end{lemma}

\begin{proof}
  Lemma \ref{rest_bound} gives us a uniform bound on $|\Rm|$ and its derivatives on $\rest_{r > br_*}$, which implies that in a small time interval the metric and its derivatives, in some fixed coordinate system, can only change so much there.  In particular we can write $g_\omega(t)$ near $r_*$ as
  \begin{align}
    \label{eq:68}
    g_\omega(t) &= \left(\frac{1}{v_{init}(r, 0)} + \alpha(r, t)\right) dr^2 \\
    &+ (h_{init}(r, 0) + \beta(r, t)) g_{S^p}
    + (r^2 + \gamma(r, t))g_{S^q},
  \end{align}
  with $r$ signifying a fixed coordinate, and where $\alpha$, $\beta$, and $\gamma$ ae zero at $t=0$.  By Lemma \ref{rest_bound}, for small time $\alpha$, $\beta$, and $\gamma$ will be small with small derivatives, independently of $\omega$.    At $r = r_*$, the functions $v(r,t)$ and $h(r,t)$ depend in smooth ways on $\alpha$, $\beta$, and $\gamma$. By choosing $T_0$ small enough, $v$ and $h$ will not have enough time to cross the barriers.
\end{proof}

\subsection{Uniform curvature bounds and convergence}\label{convergence}

\begin{lemma}\label{uniform_exist}
  There is $T_0$ independent of $\omega$ such that $T_\omega \geq T_0$.  
\end{lemma}
\begin{proof}
  Ricci flow exists up to a time when $|\Rm|$ goes to infinity.  First, suppose $T_\omega$ is smaller than the time when the functions $v$ and $h$ are trapped between barriers, and we will bound $|\Rm|$.  Here we are not worried about a uniform with respect to $\omega$, only about getting some bound if $T_\omega$ is small.

  By Lemma \ref{cross_time_bound}, the metric coefficients do not cross the barriers at $r = r_*$, so by the avoidance principle they are trapped between the barriers on $\north_{r < r_*}$.  If $T_\omega$ is smaller than the time when the coefficients are trapped, then we can bound $|\Rm|$ by using Lemma \ref{regularity_tip} in $\north_{r < r_*}$, Lemma \ref{regularity_outer} in $\north_{ar_* < r < br_*}$, and Lemma \ref{rest_bound} in $\rest_{r > ar_*}$.  
\end{proof}

\begin{lemma}
  There is a $T_0$ independent of $\omega$ so that we have the following curvature bounds for $t < T_0$.

For any $r_1>0$ and $k \in \nats \cup \{0\}$ there is $C_k(r_1) < \infty$ such that
  \begin{align}
    \label{eq:8}
    \sup_{t \in [0, T_0]} \quad \sup_{p \in \rest_{r> r_1}} |\nabla^k \Rm_{g_\omega(t)}|(p) < C_k(r_1).
  \end{align}

  For any $t_1>0$ and $k \in \nats \cup \{0\}$ there is $\tilde C_k(t_1)< \infty$ such that
  \begin{align}
    \label{eq:15}
    \sup_{t \in [t_1, T_0]} \quad \sup_{p \in M} |\Rm_{g_\omega(t)}|(p) < \tilde C_k(t_1)
  \end{align}
\end{lemma}

\begin{proof}
  The bounds for $r > r_1$ follow the same path as the bounds for $r > r_*$ in Section \ref{uniform_crossing_time}.  Briefly: bounds on $|\Rm|$ in $[(1/2)r_1, (3/4)r_1]$ follow from the parabolic regularity in Lemma \ref{basic_strip_bound}, then bounds  on $\rest_{r > r_1}$ follow by the maximum principle and Shi's estimates as in Lemma \ref{rest_bound}.

  Now we explain the bounds for $t > t_1$.  When  $t\geq t_1/2$ the functions $v_\omega(t)$ and $h_\omega(t)$ are between the barriers evaluated at time $t \geq t_1/2 + \omega \geq t_1/2$. For positive time the barriers force the functions to behave well, i.e. $v_\omega$ and $h_\omega$ satisfy the bounds required in Lemma \ref{regularity_tip}.  Therefore we have $|\Rm|$ bounds for $t \geq (3/4)t_1$ (in $\north_{r<r_*}$ by Lemma \ref{regularity_tip} and outside of $\north_{r<r_*}$ by the previous paragraph) and then Shi's global estimates \cite{Shi} imply bounds on derivatives of curvature for $t \geq t_1$.
\end{proof}

\begin{lemma}\label{convergence_lemma}
  A subsequence of the mollified solutions $g_\omega$ converge to a solution $g_*$ of Ricci flow on $C^\infty_{loc}(M \times [0, T_0] - \north_{s = 0} \times \{0\})$.  The solution $g_*$ agrees with $g_{init}$ at $t=0$ and is smooth for $t > 0$. 
\end{lemma}
\begin{proof}
  This Lemma follows from compactness, and the previous lemmata which tell us that the solutions are reasonable enough.  There are a couple of technicalities which come up from having to choose coordinates.  The full proof is the same as in Theorem 5.10 of \cite{recover} or Lemmas 12 and 13 of \cite{ACK}.
\end{proof}

\section{Formal Calculations}\label{formal_section}
\subsection{Formal solution in the outer region}\label{formal_outer}
As an initial approximation to the forward evolution, we may use
\begin{align*}
  v_{out} &\approx v_{0} + t \fopv(v_{0}, h_{0}), \\
  h_{out} &\approx h_{0} + t \foph(v_{0}, h_{0}).
\end{align*}

We calculate some derivatives of $v_0$ and $h_0$.  All these calculations use the simple fact that as $r \searrow 0$,
\begin{align}
  \partial_r \left( r^{p}|\log r|^q \right) = (1 + o(1))\cdot
  \begin{cases}
    pr^{p-1}|\log r|^q & p \neq 0 \\
    -qr^{p-1} |\log r|^{q-1} & p = 0
  \end{cases}
  \label{log_calculus}
\end{align}

\begin{align*}
  v_{0} &= k^2 |\log r|^{-1} & h_{0} &= \frac{1}{k^2}r^2 |\log r| \\
  v_{0,r} &= k^2 r^{-1}|\log r|^{-2}  & h_{0, r} &= \frac{1}{k^2}2r|\log r|(1 + o(1)) \\
  v_{0,rr} &= -k^2 r^{-2}|\log r|^{-2} (1 + o(1)) & h_{0, rr} &= \frac{1}{k^2}2 |\log r| (1 + o(1))
\end{align*}

Now calculate $\fopv[v_{0}, h_{0}]$ and $\foph[v_{0}, h_{0}]$.  First, the leading term in 
\begin{align*}
\fopv[v, h] = v v_{rr} - \oh v_r^2 -r^{-1}vv_r + \alpha_q^2 r^{-2}v(1-v) + \oh \alpha_q^2 r^{-1}v_r - \oh  p\frac{v}{h}\left( v \frac{h_r^2}{h}    \right)
\end{align*}
is
\begin{align}
  \alpha_q^2 r^{-2}v = \alpha_q^2 k^2 r^{-2}|\log r|^{-1} = \alpha_q^2 r^{-2}v_0(r) .
\end{align}
All other terms are $O(r^{-2}|\log r|^{-2})$ or smaller.

The leading term in 
\begin{align*}
  \foph
  &= v h_{rr} - v \frac{h_r^2}{h} + r^{-1}(q-1) h_r + r^{-1}vh_r - 2(p-1)
\end{align*}
is
\begin{align}
  \oh\alpha_q^2 r^{-1} h_r = \frac{1}{k^2}(q-1)2\log r = \alpha_q^2 r^{-2}h_0(r).
\end{align}
All other terms are $O(1)$ or smaller.

Therefore, we take our initial approximation to be,
\begin{align}
  v_{out}
  &= k^2 \frac{1}{-\log r} + t 2(q-1)k^2\frac{1}{r^2(-\log r)} \nonumber\\
  &= v_{init}(1 + \alpha_q^2tr^{-2}) \label{vout_recursive}\\
  &\\
  h_{out}
  &= \frac{1}{k^2}r^2(-\log r) + t\frac{1}{k^2}(q-1)2\log r \nonumber\\
  &= h_{init}(1 + \alpha_q^2tr^{-2}) \label{hout_recursive}
\end{align}
The similarity between \eqref{vout_recursive} and \eqref{hout_recursive} has an explanation.  The leading terms in both evolutions both come from (after tracing back coordinate changes) the reaction term in the evolution for $\phi$, i.e. the force of the $S^q$ factor trying to shrink.  It is not a surprise that this is the leading term; near the singular point the curvature of the $S^q$ factor is large and larger than that of the $S^p$ factor.

As long as $t \ll r^2$ we will be able to look at $h_{out}$ and $v_{out}$ as small perturbations of $h_0$ and $v_0$, and our calculations of $\fopv$ and $\foph$ will still be valid.  In Section \ref{formal_parabolic} we will look at the space-time region when $(r,t)$ goes to $(0,0)$ with $t \sim r^2$.

\subsection{Formal solution in the parabolic region}\label{formal_parabolic}
Our approximation in the outer region (which was a linearization in time) depends on the time derivative of $v$ and $h$ not changing too much.  The approximations $v_{out}$ and $h_{out}$ are
\begin{align}
  \label{eq:37}
  v_{out} = (1 + o(1))v_{init}, \quad h_{out} = (1 + o(1)) h_{init}
\end{align}
as long as $r^2 \ll t$. 
We use the parabolic coordinates,
\begin{align}
  \label{eq:32}
  \tau = \log t, \quad \rho = e^{-\tau/2}r, \quad \eta = e^{-\tau} h.
\end{align}
to study the region $r^2 \sim t$, i.e. $\rho \sim 1$.

In these coordinates, $v_{out}$ and $h_{out}$ are
\begin{align}
  \label{eq:36}
  v_{out} &= \frac{k^2}{|\log \rho + \oh  \tau|} (1 + \alpha_q^2\rho^{-2}),  \\
  \eta_{out}
          &= \frac{\rho^2}{k^2}|\log \rho + \oh \tau|(1 + \alpha_q^2 \rho^{-2})\\
          &= \frac{|\log \rho + \oh  \tau|}{k^2}(\rho^2 + \alpha_q^2).
\end{align}
So if we keep $\rho$ in a fixed region and send $\tau$ to $-\infty$,
\begin{align}
  \label{eq:38}
  v_{out} &= (1 + o(1)) \tau^{-1}2k^2  (1 + \alpha_q^2 \rho^{-2} ), \\
  \eta_{out} &= (1 + o(1)) \tau\frac{1}{2k^2}(\rho^2 + \alpha_q^2).
\end{align}
Inspired by this, we introduce $\tilde v$ and $\tilde \eta$ as
\begin{align}
  \label{eq:39}
  v = |\tau|^{-1} \tilde v, \qquad \eta = |\tau| \tilde \eta.
\end{align}
We hope we can find solutions where $\tilde v$ and $\tilde \eta$ stay bounded.  The evolution of $\tilde v$ and $\tilde \eta$ is
\begin{align}
  \label{eq:21}
  \partial_\tau \fixrho \tilde v
  &=   \oh \rho  \tilde v_\rho + \alpha_q^2 \rho^{-2} \tilde v + \oh \alpha_q^2 \rho^{-1} \tilde v_\rho \\
  &+ |\tau|^{-1} \left(
     \tilde v
    + \tilde v \tilde v_{\rho \rho} - \oh \tilde v_\rho^2 - \rho^{-1} \tilde v \tilde v_{\rho} - \alpha_q^2 \rho^{-1} \tilde v^ 2
    - \oh \tilde v^2 \frac{\tilde \eta_{\rho}^2}{\tilde \eta^2}
    \right),\\
  &\\
  \partial_\tau \fixrho \tilde \eta
  &=   \oh \rho \tilde \eta_\rho - \tilde \eta + \oh \alpha_q^2 \rho^{-1}\tilde \eta_\rho \\
  &+ |\tau|^{-1}\left( 
    -\tilde \eta
    + \tilde v \tilde \eta_{\rho\rho} - \tilde v \frac{\tilde \eta_\rho^2}{\tilde \eta} + \tilde v \rho^{-1}\tilde \eta_\rho
    - \alpha_p^2
    \right).
\end{align}
The functions
\begin{align}
  \label{eq:40}
  \tilde v_{para} &= 2k^2 (1 + \alpha_q^2 \rho^{-2}) \\
  \tilde \eta_{para} &= \frac{1}{2k^2} (\rho^2 + \alpha_q^2)
\end{align}
are steady-state solutions to the $\tau \to -\infty$ limit of these evolution equations:
\begin{align}
  \label{eq:41}
  0 &= \oh \rho  \tilde v_{para,\rho} + \alpha_q^2 \rho^{-2} \tilde v_{para} + \oh \alpha_q^2 \rho^{-1} \tilde v_{para,\rho} \\
  0 &= \oh \rho \tilde \eta_{para,\rho} - \tilde \eta_{para} + \oh \alpha_q^2 \rho^{-1}\tilde \eta_{para,\rho} \\
\end{align}
Therefore it is consistent to assume that the approximations
\begin{align}
  \label{eq:42}
  \tilde v &\sim \tilde v_{para} \\
  \tilde \eta &\sim \tilde \eta_{para}
\end{align}
are valid where $\rho \sim 1$, up to errors of order $|\tau|^{-1}$.  Note that for these approximations, the evolution equation for $\tilde v$ has error terms of order $|\tau|^{-1}\rho^{-2}$.  So these approximations may only be expected to work when $\zeta^2 \defeq \rho^2 |\tau|$ stays large.  In Section \ref{formal_tip} we will look at the space-time region when $(\rho, \tau)$ goes to $(0, -\infty)$ with $|\tau| \sim \rho^{-2}$.

\subsection{Formal solution in the tip region}\label{formal_tip}
Here we study the equations when $\zeta = \sqrt{|\tau|}\rho \sim 1$.  Let
\begin{align*}
  \zeta^2 = \frac{r^2}{\theta} = |\tau|\rho^2, \quad \theta = \frac{t}{-\log t}, \quad H = \frac{h}{\theta} = |\tau| \eta = |\tau|^2 \tilde \eta.
\end{align*}
Put the parabolic approximation $v_{para}$ in these coordinates.
\begin{align}
  \label{eq:43}
  v_{para} &= |\tau|^{-1} \tilde v_{para} \\
                  &= 2k^2 \alpha_q^2 \zeta^{-2} +  2k^2 |\tau|^{-1} \\
  H_{para} &= |\tau|^2 \tilde \eta_{para} \\
                  &= \frac{1}{2k^2}|\tau| \zeta^2 + \frac{\alpha_q^2}{2k^2} |\tau|^2 
\end{align}
Define $\tilde H = |\tau|^{-2} H$, which we will hope to be bounded in the tip region.
\begin{align}
  v_{para} &= 2k^2 \alpha_q^2 \zeta^{-2} + 2k^2 |\tau|^{-1}, \label{tildev_para}\\
  \tilde H_{para} &= \frac{\alpha_q^2}{2k^2} + \frac{\zeta^2}{2k^2} |\tau|^{-1}. \label{tildeH_para}
\end{align}

Then we can calculate the evolutions in terms of $\zeta$, $\tilde H$, and $v$.  Using the chain rule,
\begin{align}
  \partial_t \fixr v
  &= (\partial_\zeta v)\partial_t \fixr \zeta + \partial_t \fixz v \\
  &= (- \oh \theta^{-3/2} \theta_t r) \partial_\zeta v + \tau_t \partial_\tau \fixz v \\
  &= - \oh \theta^{-1}\theta_t \zeta \partial_\zeta v + \theta^{-1}|\tau|^{-1} \partial_\tau \fixz v
\end{align}
Then, calculating $\theta_t$ we find
\begin{align}
  \partial_t \fixr v
  &= \theta^{-1}|\tau|^{-1} \left( - \oh \zeta \partial_\zeta v +  \partial_\tau \fixz v \right)
  + \theta^{-1}|\tau|^{-2}\left[  \oh \zeta \partial_\zeta v \right].
\end{align}
For this section we follow a convention of using square brackets for terms which will always be negligible.  Annoying negligible terms will come from the error in the approximations $1/\theta_t \approx |\log \theta| \approx |\tau|$.  Similarly calculate for $h$, and then replace it with $\theta \tau^2 \tilde H$.
\begin{align}
  \partial_t \fixr h
  &= \theta^{-1}\left(- \oh  \theta_t \zeta \partial_\zeta h +  |\tau|^{-1} \partial_\tau \fixz h  \right)\\
  &=\theta^{-1}
    \left(
    - \oh \theta_t \zeta \partial_\zeta (\theta \tau^2 \tilde H)
     + |\tau|^{-1}\partial_\tau \fixz (\theta \tau^2 \tilde H)
     \right)
     \\
  &= 
    \left(
    - \oh ( \theta_t \tau^2) \zeta \partial_\zeta \tilde H
    + \left( \theta^{-1}\theta_\tau \tau  + 2   \right)\tilde H
    +  \tau \partial_\tau \tilde H
    \right)\\
  &= |\tau|
    \left(
    - \oh \zeta \partial_\zeta \tilde H
    + \tilde H
    + \partial_\tau \tilde H
    \right)
    +
    \left[
    - \oh \zeta \partial_\zeta \tilde H
    + \tilde H
    \right]
\end{align}
The right-hand sides scale as follows.
\begin{align}
  \fopv[v, h]
  &= \theta^{-1} \fopv_\zeta [v. H] \label{fopv_scaling_H}\\
  &= \theta^{-1} \fopv_\zeta [v, \tilde H]\label{fopv_scaling_tildeH}
\end{align}
\begin{align}
  \foph_\zeta[v,h]
  &= \foph_\zeta[v, H] \label{foph_scaling_H}\\
  &= |\tau|^2 \left(
    v \tilde H_{\zeta \zeta} - v \frac{\tilde H_\zeta^2}{\tilde H} + \zeta^{-1}\oh \alpha_q^2 \tilde H_\zeta + \zeta^{-1} v \tilde H_\zeta
    -|\tau|^{-2}\alpha_p^2
    \right) \label{foph_scaling_tildeH}
\end{align}

Here, by $\foph_\zeta$ and $\fopv_{\zeta}$ we mean the operators with $r$ and $\partial_r$ replaced with $\zeta$ and $\partial_\zeta$.  The scaling in
\eqref{fopv_scaling_H}, \eqref{fopv_scaling_tildeH}, \eqref{foph_scaling_H}, \eqref{foph_scaling_tildeH}
are simple, but we can also understand them as follows.  The scaling
\begin{align}
  (r, v, h) \to (\zeta, v, H)  
\end{align}
is a vanilla scaling of the metric by a factor of $\theta^{-1}$.  The quantity $v$ is scale invariant, and $h$ scales in the same way as the metric.  The fact that time scales like the metric under Ricci flow (so time derivatives scale like one on the metric) explains \eqref{fopv_scaling_H} and \eqref{foph_scaling_H}.  On the other hand, the scaling
\begin{align}
  (\zeta, v, H) \to (\zeta, v, \tilde H)
\end{align}
is not a straightforward scaling of the metric.  What's happening is that the $S^p$ factor is much larger than the $\real \times S^q$ part, and is trying to become an $\real^p$ (as we go backwards in time).  A nice way to think of this scaling by shifting the $\tau^2$ factor to the radius of the $S^p$.  In other words, consider that
\begin{align}
  \frac{1}{v}d\zeta^2 + H g_{S^p} + \zeta^2 g_{S^q} = \frac{1}{v} d\zeta^2 + \tilde H \left( \tau^2 g_{S^p}\right) + \zeta^2 g_{S^q}
\end{align}
So $\tilde H$ corresponds to considering the $S^p$ factor to have a radius of $|\tau|$.  This explains \eqref{foph_scaling_tildeH}; as the radius gets large, the Ricci curvature gets small compared to the metric so the reaction term shrinks.

So performing a cancellation, the evolutions are as follows:
\begin{align}
  |\tau|^{-1} \left(-\oh  \zeta \partial_\zeta v +  \partial_\theta \fixz v \right)
  ={}
  & \lop[v] + \qop[v,v] - 2p |\tau|^{-2}\frac{v}{\tilde H} (\psi_s)^2 \label{evo_v_tip}\\
  &-|\tau|^{-2}\left[  \oh \zeta \partial_\zeta v \right]
\end{align}
\begin{align}
  |\tau|^{-1}
  \left(
  - \oh  \zeta \partial_\zeta \tilde H
  +  \tilde H
  + \partial_\theta \fixz \tilde H
  \right)
  ={}
  & \rop[v, \tilde H] - v \frac{\tilde H_\zeta^2}{\tilde H} - \tau^{-2}\alpha_p^2 \label{evo_H_tip} \\
  &-
  |\tau|^{-2}
  \left[
  - \oh \zeta \partial_\zeta \tilde H
  + \tilde H
  \right] 
\end{align}
Where we have defined,
\begin{align*}
  \lop[v] &= \alpha_q^2\zeta^{-2}v + \oh\alpha_q^2 r^{-1}v_r \\
  \qop[v,w] &= \oh (vw_{\zeta\zeta} + wv_{\zeta\zeta}) - \oh v_rw_r - \zeta^{-1}(vw_\zeta + wv_\zeta) - \alpha_q^2 \zeta^{-2} vw \\
  \rop[v, \tilde H] &= v \tilde H_{\zeta\zeta} + \zeta^{-1}\left( \oh \alpha_q^2 + v \right) \tilde H_v
\end{align*}

Assume that $v$ and $\tilde H$ are bounded (in $C^2$) in regions where $\zeta$ is bounded, up to time $0$.  Then taking the $\tau \searrow -\infty$ limit
of the equations above:
\begin{align}
  0 &= \fop_{v, \zeta}[v_0, \tilde H_0] \\
  0 &= \rop[v, \tilde H] - v_0 \frac{\tilde H_\zeta^2}{\tilde H}
\end{align}
These are the equations for a steady soliton of the form
\begin{align*}
  \frac{1}{v_0} d\zeta^2 + \tilde H_0 g_{\real^p} + \zeta^2 g_{S^q},
\end{align*}
which one can see by realizing that these right-hand-sides can be obtained by setting the Ricci curvature of the $S^p$ factor to be zero.  In order to match the value of $\tilde H$ in the outer region at $\theta = 0$ \eqref{tildeH_para}, $\tilde H_0$ should approach the constant $\frac{\alpha_q^2}{2k^2}$ at infinity.  By Lemma \ref{eternal_constant}, $\tilde H_0$ is constant.  Then, we find that the steady soliton is $(\text{a }(q+1)\text{-dimensional Bryant Soliton}) \times (\text{flat } \real^P)$.  We find it convenient to introduce $c_0^2 = \frac{1}{2\alpha_q^2k^2}$ here, so $\tilde H_0$ approaches $\alpha_q^4 k_0$.

For any $c$, a Bryant soliton is given by
\begin{align}
  \label{eq:50}
  v(\zeta) = \bry\left(  c\zeta  \right)
\end{align}
where $\bry:[0, \infty) \to (0,1]$ is a fixed function.  This has the asymptotics
\begin{align}
  \label{eq:52}
  \bry\left(c\zeta \right) \sim c^{-2}\zeta^{-2} \text{ at infinity},
  \qquad  \bry\left( c\zeta \right) \sim 1 - b_0^2c^2 \zeta^2 \text{ at zero, for some }b_0 > 0
\end{align}
By matching the parabolic approximation \eqref{tildev_para}, \eqref{tildeH_para} as $\zeta \to \infty$, we guess
\begin{align}
  \label{eq:49}
  v_0(\zeta) = \bry\left(c\zeta \right) , \quad \tilde H_0(\zeta) = \alpha_q^4 c, \qquad \text{with } c = k_0.
\end{align}

We look for the next order terms, which are not constant in time.  The evolution equations \eqref{evo_v_tip}, \eqref{evo_H_tip} will have terms of order $|\tau|^{-1}$.  This is also consistent with the parabolic approximation \eqref{tildev_para},\eqref{tildeH_para}.  Write $v = v_0 + |\tau|^{-1}v_1 + o(|\tau|^{-1})$ and $\tilde H = \tilde H_0 + |\tau|^{-1}H_1 + o(|\tau|^{-1})$, insert into \eqref{evo_v_tip}, \eqref{evo_H_tip}, and find that $v_1$ and $\tilde H_1$ must satisfy,
\begin{align}
  - \oh \zeta \partial_\zeta v_0
  &= 2\qop[v_0, v_1] + \lop[v_1] \label{eqn_for_v1}\\
  &\\
  - \oh \zeta \partial_\zeta \tilde H_0 + \tilde H_0
  &= \rop[v_0, \tilde H_1] \label{eqn_for_H1}
\end{align}
Note that the coupled term $2p \frac{v}{h} \psi_s^2$ does not appear in the first equation because $H = O(|\tau|^{2})$.  We find the following about the solutions:

\begin{lemma}
  The equation \eqref{eqn_for_v1} with $v_0(\zeta) = \bry(\zeta)$ has a one-parameter family of strictly positive solutions which vanish at zero.  Let $\cry$ be any of them.  The function $\cry$ satisfies
  \begin{align}
     \cry(\zeta) = \frac{1}{\alpha_q^2} + o(1) \text{ at infinity },\qquad \cry(\zeta) = O(\zeta^2) \text{ at zero.}
  \end{align}
  Let $c$ be given.  Then a solution to \eqref{eqn_for_v1} with $v_0(\zeta) = \bry(c\zeta)$ is given by $\cry(\zeta) = c^{-2}\cry(c\zeta)$ and has asymptotics
  \begin{align}
    c^{-2}\cry(c\zeta) = \frac{1}{c^2\alpha_q^2} + o(1) \text{ at infinity },\qquad \cry(\zeta) = O(\zeta^2) \text{ at zero.}
  \end{align}

  Now consider the equation \eqref{eqn_for_H1} with $v_0 = \bry(\zeta)$, $\tilde H_{0} = \alpha_q^4$, and the boundary condition that $\tilde H_1$ and its derivative vanish at zero.  This has one solution $\ary(\zeta)$ which has the asymptotics
  \begin{align*}
    \ary(\zeta) = \alpha_q^2\zeta^2 + o(\zeta^2) \text{ at infinity },\qquad \ary(\zeta) = O(\zeta^2) \text{ at zero.}
  \end{align*}
  The solution to \eqref{eqn_for_H1} with $v_0 = \bry(c\zeta)$,  $\tilde H_{0} = c^2 \alpha_q^4$, and the same boundary conditions is given by $\tilde H_{1} = \ary(c\zeta)$ which has asymptotics
  \begin{align*}
    \ary(c\zeta) = \alpha_q^2c^2\zeta^2 + o(\zeta^2) \text{ at infinity },\qquad \ary(\zeta) = O(\zeta^2) \text{ at zero.}
  \end{align*}
\end{lemma}
\begin{proof}
  The equation for $v_1$ \eqref{eqn_for_v1} is essentially same as appears in Lemma 4 of \cite{ACK}.  This is because the coupling term makes no appearance.  There are two things we must note.  One is that the $(n-1)$ in the asmptotics in Lemma 4 of \cite{ACK} should read $\frac{1}{n-1}$, which corresponds to $\frac{2}{\alpha_q^2}$ for us.  The other difference is that we have put $\theta = \frac{t}{-\log t}$ rather than its square root; this leads to a factor of two difference in our equation \eqref{eqn_for_H1} compared to (4.22) of \cite{ACK}.  This accounts for the asymptotics $\frac{1}{\alpha_q^2}$ at infinity for $\cry(\zeta)$.

  The equation for $\tilde H_1$ \eqref{eqn_for_H1} can be solved semi-explicitly.  With $\tilde H_{0} = \alpha_q^4$ and $v_0 = \bry(\zeta)$ the equation is
  \begin{align*}
    \partial_\zeta^2\tilde H_{1} + \frac{\oh \alpha_q^2  + \bry(\zeta)}{\zeta \bry(\zeta)} \partial_\zeta \tilde H_1
    = \frac{\alpha_q^4}{\bry}.
  \end{align*}
Set $Q(\zeta) = \int_1^\zeta \frac{\oh \alpha_q^2 + \bry(z)}{z \bry(\zeta)} dz$, then a solution is 
\begin{align}
  \label{eq:47}
  \ary(\zeta) = \alpha_q^4\int_0^\zeta \left( e^{ -Q(w)} \int_0^w \frac{e^{Q(z)}}{\bry(z)}dz \right) dw.
\end{align}
Using the known asymptotics for $\bry(\zeta)$ one can check that the integrals are well defined and find the asymptotics of $\ary$.  The case for $\tilde H_0(\zeta) = c\alpha_q^4$ and $v_0(\zeta) = \bry(c\zeta)$ is a straightforward scaling.
\end{proof}

Finally, we see that with $c_0 = \frac{1}{\sqrt{2}\alpha_q k}$ the functions
\begin{align*}
  v(\zeta) &= \bry(c_0 \zeta) + c_0^{-2}|\tau|^{-1} \cry(c_0\zeta)  \\
  \tilde H_1(\zeta) &= c_0^2 \alpha_q^4 + |\tau|^{-1}\ary(c_0\zeta)
\end{align*}
solve \eqref{evo_v_tip},\eqref{evo_H_tip} up to order $|\tau|^{-2}$.

\section{Barriers}\label{barriers_section}
In this section we use the results of Section \ref{formal_section} to construct barriers according to Definition \ref{barrier_def}, proving Lemma \ref{barriers_full}.

An important observation from our formal calculations is that the size of $\psi_s^2 = 4\frac{v h_r^2}{h}$ is not very important.  It is only important that we have the bound $\psi_s^2 \leq 1$.  Therefore, for creating barriers we can switch to the notation \eqref{gopv_def}, \eqref{goph_def}.  This is in contrast to $v = \phi_s^2$; it is important that this becomes 1 near the tip, as this makes the evolution equation in the $r$ coordinate strictly parabolic.  

For analysis it is convenient to look at $\gopv$ and $\goph$ as
\begin{align}
  \label{eq:97}
  r^2 \gopv(v, h, u)
  &= r^2 \lop (v) + r^2 \qop (v, v) - 2 p u v \frac{r^2}{h} \\
  r^2 \goph(v, h, u)
  &= \rop(v, h) - u - \alpha_p^2
\end{align}
where
\begin{align}
  \label{eq:104}
  r^2 \lop (v) &= \oh \alpha_q^2 r \partial_r v + \alpha_q^2 v \\
  r^2 \qop (v, v) &= v (r^2 \partial^2_r v) - \oh r^2 (\partial_r v )^2 - v (r\partial_r v) - \alpha_q^2 v^2  \\
  r^2 \rop (v, h) &= v (r^2 \partial^2_r h) + (\oh \alpha_q^2 + v)(r \partial_r h)
\end{align}

For a function $f$ of $r$ let
\begin{align}
  \label{eq:71}
  |f|_{2,r}(r) = |f(r)| + |r\partial_r f(r)| + |r^2\partial_r^2 f(r)|.
\end{align}
For many functions we consider we will have
\begin{align}
  |f|_{2, r}(r) = O_{r \to 0}(|f|(r)) \quad \text{and} \quad |f|_{2, r}(r) = O_{r \to \infty}(|f|(r)).
\end{align}

Notice that for any $u$ fixed, $r^2\gopv(v, h, u)$ is a quadratic polynomial of
\begin{align}
  \label{eq:91}
  v, \quad r\partial_r v, \quad r^2 \partial_r v, \quad \frac{r^2}{h};
\end{align}
and $r^2 \goph(v, h, u)$ is a quadratic polynomial of
\begin{align}
  \label{eq:93}
  h, \quad r\partial_r h, \quad r^2 \partial_r^2 h, \quad v, \quad \frac{r^2}{h}.
\end{align}

We will use that, for any $C_1$, there is a constant $C_2$ so that for all $v_2$ and $h_2$ with 
\begin{align}
  \label{eq:100}
  |v_1|_{2, r} + |v_2|_{2,r} \leq C_1  \quad |h_1|_{2, r} + |h_2|_{2, r} \leq C_1 \\
  \frac{r^2}{|h_1|} + \frac{r^2}{|h_2|} \leq C_1,
\end{align}
we have the pointwise bounds
\begin{align}
  \label{out_fopv_bnd}
  |r^2\gopv(v_2, h_2, u_2) - r^2\gopv(v_1, h_1, u_1)| \leq C_2 |v_2 - v_1|_{2, r}
\end{align}
and
\begin{align}
  \label{out_foph_bnd}
  |r^2\goph(v_2, h_2, u_2) - r^2\goph(v_1, h_1, u_1)| \leq C_2 |h_2 - h_1|_{2, r}.
\end{align}
If we keep $u_i \in [0, 1]$ then $C_2$ is independent of $u_1, u_2$.

\subsection{Barriers in the outer region}
Our outer approximation is a simple linearization in time:
\begin{align}
  \label{eq:105}
  v_{out} \approx v_0 + t \fopv(v_0, h_0), \quad h_{out} \approx h_0 + t \foph (v_0, h_0)
\end{align}
Therefore creating barriers boils down to checking how long the right hand side of the evolution equation does not change too much.  Then since both of the right-hand sides are positive, one can add a term of the form $\pm \epsilon \fop(v_0, h_0)$ to the approximation to get a barrier.
\begin{lemma}\label{outer_barriers}
  Let $\delta, \epsilon > 0$ be given and sufficiently small, and let
  \begin{align}
    \label{eq:60}
    a^2_{\pm} = (1 \pm \delta)k^2, \quad b^2_{\pm} = (1 \pm \delta) k^{-2}, \quad d_\pm = (1 \pm \epsilon).
  \end{align}
  Let
  \begin{align}
    \label{eq:24}
    v_{out}^{(\pm)}(r, t) &= a^2_\pm|\log r|^{-1}\left( 1 + d_\pm  \alpha_q^2 tr^{-2}\right), \\
    h_{out}^{(\pm)}(r, t) &= b^2_\pm r^2 |\log r|\left( 1 + d_\pm \alpha_q^2 tr^{-2} \right).
  \end{align}
  Then there are $\rho_1, r_*$ such that $v_{out}^+$, $v_{out}^-$, $h_{out}^+$, $h_{out}^-$ are barriers for $r \in [\sqrt{t}\rho_*, r_*]$ (and $t<\frac{r_*}{\rho_*}$).

Furthermore, if $\delta$ and $\epsilon$ are uniformly small,  there is $A_*(k, p, q)$ so that $\rho_1 > \epsilon^{-1/2}A_*$ suffices.
\end{lemma}
\begin{proof}

  Here let $v_0 = a^2_\pm |\log r|^{-1}$ and $h_0 = b^2_\pm r^2 |\log r|$.  Check that (mostly by \eqref{log_calculus})
  \begin{align}
    \label{out_diff_bnds}
    |v_{out}^\pm-v_0|_{2, r}\leq C \rho^{-2}v_0.
  \end{align}
  and
  \begin{align}
    \label{eq:98}
    |h_{out}^\pm-h_0|_{2, r} \leq C \rho^{-2}h_0
  \end{align}
  Therefore the bounds \eqref{out_fopv_bnd} and \eqref{out_foph_bnd} tell us that for any $h^*$ between $h_{out}^{\pm}$ and $v^*$ between $v_{out}^\pm$ we have
  \begin{align}
    \left| r^2\gopv(v_0, h_0, u) - r^2\gopv(v_{out}^{\pm}, h^*, u) \right| \leq C \rho^{-2} v_0, \label{v0_fopdiff_bnd}\\
    \left| r^2\goph(v_0, h_0, u) - r^2\goph(v^*, h_{out}^{\pm}, u) \right| \leq C \rho^{-2} h_0. \label{h0_fopdiff_bnd}\
  \end{align}  
  Our formal calculation showed that for all $u \in [0,1]$,
  \begin{align}
    r^2 \gopv(v_0, h_0, u) = (1 + O(|\log r|^{-1})) \alpha_q^2 v_0, \label{formal_calc_v0}\\
    \quad
    r^2 \goph(v_0, h_0, u) = (1 + O(|\log r|^{-1})) \alpha_q^2 h_0. \label{formal_calc_h0}
  \end{align}

  We show that $\gopv(v^\pm_{out}, h^*, u)$ has the correct sign, i.e. that we have \eqref{subsuper_v_ineq}. Calculate,
  \begin{align}
    \label{eq:99}
    r^2\left( \partial_t v_{out} - \gopv(v_{out}, h^*, u) \right)
    &= \left( (1 \pm \epsilon)\alpha_q^2 v_0  - r^2 \gopv(v_0, h_0, u) \right)\\
    &+ \left( r^2\gopv(v_0, h_0, u) - r^2\gopv(v_{out}, h^*, u) \right)\\
  \end{align}
  So by \eqref{v0_fopdiff_bnd} and \eqref{formal_calc_v0},
  \begin{align}
    \label{eq:109}
    r^2\left( \partial_t v_{out} - \gopv(v_{out}^\pm, h^*, u) \right)
    &= \left(\pm \epsilon + O_{r \to 0}(|\log r|^{-1}) + O_{ \rho \to \infty}(\rho^{-2}) \right) \alpha_q^2 v_0
  \end{align}
  Therefore we can choose $\rho_1$ large enough and $r_*$ small enough so that the $\pm \epsilon$ term dominates above.  Showing that $\goph(v_{out}^*, h^\pm_{out}, u)$ has the correct sign is a similar calculation using \eqref{h0_fopdiff_bnd} and \eqref{formal_calc_h0}.
  
\end{proof}

\subsection{Barriers in the parabolic region}
Recall our parabolic coordinates.
\begin{align}
  \rho = \frac{r}{\sqrt t}, \quad \tau = \log t , \quad \eta = \frac{h}{r^2}
\end{align}
Our formal calculations show that in the parabolic region the functions
\begin{align}
  \tilde v = |\tau|^{-1} v, \quad \tilde \eta = |\tau| \eta 
\end{align}
are approximated by
\begin{align}
  \tilde v_{para} = 2k^2 (1 + \alpha_q^2 \rho^{-2}), \quad \tilde \eta_{para} = \frac{1}{2k^2}(\rho^2 + \alpha_q^2).
\end{align}

In terms of the parabolic coordinates, $\dopv$ and $\doph$ take the form
\begin{align}
  |\tau| e^{\tau} \left(\partial_t v - \gopv(v, h, u) \right)
  &= \partial_\tau \fixrho \tilde v - \oh \rho \partial_\rho \tilde v - \lop_\rho(\tilde v) \label{dopv_para_lin}\\
  &+ |\tau|^{-1}
    \left(
    - \tilde v
    -\qop_\rho (\tilde v, \tilde v)
    + 2 p u  \frac{\tilde v}{\tilde \eta}
    \right),\label{dopv_para_quad}\\
  &\\
  |\tau|^{-1}e^{\tau} \left( \partial_t h - \goph(v, h, u) \right)
  &= \partial_\tau \fixrho \tilde \eta - \oh \rho \tilde \eta_\rho + \tilde \eta - \oh \alpha_q^2 \rho^{-1}\tilde \eta_\rho
    \label{doph_para_lin}
  \\
  &+ |\tau|^{-1}
    \left( 
    \tilde \eta
    - \tilde v \tilde \eta_{\rho\rho}  - \tilde v \rho^{-1}\tilde \eta_\rho
    + \alpha_p^2 + u
    \right) . \label{doph_para_quad}
\end{align}

\begin{lemma}\label{para_barriers}
  Let $\rho_2$ be given.  Also let $\hat a_{\pm}$ and $\hat b_{\pm}$ be given satisfying
  \begin{align}
    (1-1/2)k^2 < \hat a_-^2 \leq k^2 \leq \hat a_+^2 < (1+1/2)^2, \qquad     (1-1/2)k^{-2} < \hat b_-^2 \leq k^{-2} \leq \hat b_+^2 < (1+1/2)^{-2}.
  \end{align}
  There is a $\tau_*$, $\zeta_1$, and $D$ so that
  \begin{align}
    \label{eq:23}
    \tilde v_{para}^{(\pm)} &= 2{\hat a^2}_\pm \left( 1 + \alpha_q^2\rho^{-2} \right)  \pm D  |\tau|^{-1}\rho^{-4} \\
    \tilde \eta_{para}^{(\pm)} &= \oh {\hat b^2}_\pm  \rho^2 \left( 1 + \alpha_q^2\rho^{-2} \right)  \pm D|\tau|^{-1}\rho^{-2}  .
  \end{align}
  define barriers for $\tau < \tau_*$ and $\rho \in \left[\frac{\zeta_1}{\sqrt{|\tau|}}, \rho_2 \right]$.

  Furthermore, $\zeta_1 \geq \sqrt{D}$ is sufficiently large.
\end{lemma}
\begin{proof}
  We look at $v$ first.   We need to check that for any $u \in [0, 1]$ and $\tilde \eta^* \geq \tilde \eta_{para}^-$, \eqref{dopv_para_lin}, \eqref{dopv_para_quad} has the correct sign.  The barrier $\tilde v^\pm_{para}$ can be written as
  \begin{align}
    \tilde v_{para}^\pm = \tilde v_{para} \pm D |\tau|^{-1}\rho^{-4}
  \end{align}
  The first line \eqref{dopv_para_lin} is linear and $\tilde v_{para}$ was chosen so that this line vanishes for $\tilde v_{para}$.  If we look at this first line and plug in $\tilde v _{para}^\pm$ we find
  \begin{align}
    |\tau|\rho^6\cdot \text{Line \eqref{dopv_para_lin}}
    &= \pm D |\tau|\rho^6\left( |\tau|^{-2}\rho^{-2} + 2 \rho^{-2} - |\tau|^{-1}\lop_{\rho}(\rho^{-4}) \right)\\
    &= \pm D |\tau|\rho^6\left( |\tau|^{-2}\rho^{-2} + 2 |\tau|^{-1}\rho^{-2} +  |\tau|^{-1}\rho^{-6} \right)\\
    &= \pm D \pm D \left( |\tau|^{-1}\rho^4 + 2 \rho^4\right) \label{lin_calc_parabar}
  \end{align}
  The first term, which has the correct sign, will dominate.  The second line \eqref{dopv_para_quad} gives us
  \begin{align}
    \left| \tau \rho^6 \cdot \text{Line \eqref{dopv_para_quad}} \right|
    &\leq \left| \rho^6 \tilde v_{para}^\pm \right|
      +
      \left| \rho^4\rho^2 \qop_\rho(\tilde v_{para}^\pm, \tilde v_{para}^\pm) \right|
      +
      \left|
      \left( \frac{2pu}{ \rho^{-2}\tilde \eta_{para}^-} \right) \rho^4 \tilde v_{para}^\pm 
      \right|
  \end{align}
  It is easy to check that each term here is bounded by a polynomial in $\rho^2$ and $D|\tau|^{-1}\rho^{-2} = D \zeta^{-2}$.  (To bound the $\rho^2 \qop_\rho (\tilde v^\pm_{para}, \tilde v^\pm_{para})$ term one just needs to know that $\rho^2\qop_\rho(v, v)$ is a quadratic in $v,\; \rho\partial_\rho v,\; \rho^2 \partial_\rho^2 v$.)  The setup of the lemma gives us a bound on $\rho^2$, as well as $\zeta^2 D \leq 1$.  Therefore $|\tau \rho^6 \cdot \text{Line \eqref{dopv_para_quad}}|$ is bounded independent of $D$, so taking $D$ large enough makes the first term of \eqref{lin_calc_parabar} dominate.

  Checking $h$ is a similar calculation.  The first line \eqref{dopv_para_lin} gives us:
  \begin{align}
    |\tau|\rho^4 \cdot \text{Line \eqref{dopv_para_lin}}
    &= \pm D |\tau|\rho^4 \left( |\tau|^{-2}\rho^{-2} + |\tau|^{-1}\rho^{-4} + |\tau|^{-1}\rho^{-2} + \alpha_q^2 |\tau|^{-1}\rho^{-4} \right) \\
    &= \pm \alpha_q^2 D  \pm D \left( |\tau|^{-1}\rho^2 + 1 + \rho^2 \right).
  \end{align}
  Then we just check that
  \begin{align}
    |\tau|\rho^4 \cdot \text{Line \eqref{dopv_para_quad}}
  \end{align}
  is bounded by a polynomial in $\rho^2$ and $D |\tau|^{-1}\rho^{-2}$, and we are done.
\end{proof}

\subsection{Barriers in the tip region}
In terms of the coordinates for the tip we can write $\dopv$ and $\doph$ as
\begin{align}
  \label{eq:57}
  \theta \left( \dopv [v, h] \right)
  ={}&- \lop[v] - \qop[v,v]  \\
  &+ |\tau|^{-1}\left(-\oh \zeta \partial_\zeta v  +  \partial_\tau \fixz v \right)\\
  &+ |\tau|^{-2}\left[  \oh \zeta \partial_\zeta v + 2 p \frac{v}{\tilde H}(\psi_s^2)\right]
  &\\
  \tau^{-2} \left( \doph [v, h] \right)
  ={}&- \rop[v, \tilde H] + v \frac{\tilde H_{\zeta}^2}{\tilde H}\\
     &+ |\tau|^{-1}
       \left(
       - \oh \zeta \partial_\zeta \tilde H
       + \tilde H 
       + \partial_\tau\fixz \tilde H
       \right)  \\
     &+
       |\tau|^{-2}
       \left[
       - \oh \zeta \partial_\zeta \tilde H
       + \tilde H
       - \alpha_p^2
       \right] \\
\end{align}

The formal solutions in the tip region are given by
\begin{align*}
  v(\zeta) &= \bry(c_0 \zeta) + c_0^{-2}|\tau|^{-1} \cry(c_0\zeta),  \\
  \tilde H_1(\zeta) &= c_0^2 \alpha_q^4 + |\tau|^{-1} \ary(c_0\zeta).
\end{align*}
\begin{lemma}\label{tip_barriers}
  Let $\zeta_2 > 0$, $\epsilon > 0$ be given.
  Let $c_{\pm}, \bar c_{\pm}$ satisfy 
  \begin{align}
    (1-1/2) c_0 < c_+ <  c_0 = \frac{1}{\sqrt{2} \alpha_q k}  < c_- <  (1+1/2) c_0 \\
    (1-1/2) c_0 < \bar c_- <  c_0 = \frac{1}{\sqrt{2} \alpha_q k}  < \bar c_+ <  (1+1/2) c_0 
  \end{align}

  There exists $t_*$ so that 
  \begin{align*}
    v_{tip}^\pm (\zeta, t) &= \bry(c_\pm \zeta) + (1 \mp \epsilon) |\tau|^{-1} c_\pm^{-2} \cry(c_\pm \zeta) \\
    \tilde H_{tip}^\pm(\zeta, t) &=  \alpha_q^4 (\bar c_\pm^2) +  |\tau|^{-1} \ary(c_\pm \zeta)
  \end{align*}
  are barriers for $t < t_*$ and $\zeta \in [0, \zeta_2]$.
\end{lemma}
We make a few remarks on the choices of barriers above.  For regularity it is vital that $v$ goes to $1$ at a quadratic rate at $r=0$.  (See Lemma \ref{regularity_tip}.)  For that reason we cannot put a coefficient on the time-independent term, as we did before.  Second, note that $\bry(\zeta)$ is decreasing near zero; that is why we choose $c_+ < c_-$ but $\bar c_+ > \bar c_-$.  For gluing we actually may need to choose $\bar c_+ \neq c_-$.

\begin{proof}
  Write
  \begin{align}
    \label{eq:58}
    v_0 = \bry(c_\pm \zeta) \quad v_1 = c^{-2}_\pm\cry(c_\pm \zeta)  \\
    \tilde H_0 = \alpha_q^4 \bar c_\pm \quad \tilde H_1 = \ary(c_\pm \zeta).
  \end{align}
  First check $\dopv$.  Remember $v_0$ is such that $\lop[v_0] + \qop[v_0, v_0] = 0$, and $v_1$ is such that $-\oh \zeta \partial_\zeta v_0 = \lop[v_1] + 2 \qop[v_0, v_1]$.  Therefore,
  \begin{align}
    \label{eq:55}
    \theta \dopv [v^\pm_{tip},h^{\pm}_{tip}]
    ={}& - \lop[v_0] - \qop[v_0, v_0] \\
       &+|\tau|^{-1} \left(
         -\oh \zeta \partial_\zeta v_0 -  (1\mp \epsilon) \lop[v_1] -  (1\mp \epsilon)2\qop[v_0, v_1]
         \right)\\
       &+ |\tau|^{-2} \left(
         -\oh (1\mp \epsilon)\zeta \partial_\zeta v_1
         - |\tau|^{-1} (1\mp \epsilon)v_1
         - (1\mp \partial)\qop[v_1, v_1]
         \right)\\
       &+ |\tau|^{-2} \left(
         \oh \zeta \partial_\zeta v + 2p \frac{v}{\tilde H}(\psi_s^2)
         \right)\\
    ={}&{} |\tau|^{-1}\left(\mp \epsilon \oh \zeta \partial_\zeta v_0\right)\\
       &+|\tau|^{-1}
         \left(
         -(1 \mp \epsilon)|\tau|^{-1} \oh \zeta \partial_\zeta v_0
         - |\tau|^{-1} (1\mp \epsilon) \lop[v_1]
         - |\tau|^{-1} (1\mp \epsilon)2\qop[v_0, v_1]
         \right)\\
       &+ |\tau|^{-2}F_2(\zeta) + |\tau|^{-3}F_2(\zeta) \\
       % &+ |\tau|^{-2} \left(
       %   -\oh (1\mp \epsilon)\zeta \partial_\zeta v_1
       %   - |\tau|^{-1} (1\mp \epsilon)v_1
       %   - (1\mp \partial)\qop[v_1, v_1]
       %   \right)\\
       % &+ |\tau|^{-2} \left(
       %   \oh \zeta \partial_\zeta v + 2p \frac{v}{\tilde H}(\psi_s^2)
       %   \right)\\
    ={}& |\tau|^{-1} \left( \mp \epsilon  \oh \zeta \partial_\zeta v_0 \right) \\
       &+ |\tau|^{-2}F_2(\zeta) + |\tau|^{-3}F_2(\zeta) \\
       % &+ |\tau|^{-2} \left(
       %   -\oh (1\mp \epsilon)\zeta \partial_\zeta v_1
       %   - |\tau|^{-1} (1\mp \epsilon)v_1
       %   - (1\mp \partial)\qop[v_1, v_1]
       %   \right)\\
       % &+ |\tau|^{-2} \left(
       %   \oh \zeta \partial_\zeta v + 2p \frac{v}{\tilde H}(\psi_s^2)
       %   \right)\\
  \end{align}
  The first term has the correct sign because $v_0$ is decreasing; we want to show that it controls the rest.  The function $v_0$ has strictly negative second derivative at zero and $\partial_\zeta\bry < 0$ for $\zeta > 0$.  By continuity there is $m$ depending on $\zeta_2$ which is small enough so that $- \oh \zeta \partial_\zeta v_0 \geq m \zeta^2$ for $\zeta \in [0, \zeta_2]$.  By using the asymptotics we know for $v_0$ and $v_1$, similar continuity considerations bound
  \begin{align}
    |F_2(\zeta)| + |F_3(\zeta)| \leq M \zeta^2
  \end{align}
  for a large $M$ depending on $\zeta_2$.  By demanding that $\tau < \tau_*$ is small enough so that $|\tau|^{-1} < \frac{m\epsilon}{2M}$, we can force $\dopv$ to have the correct sign.  

  Now we check $\doph$.  Remember $\tilde H_1$ is such that $\alpha_q^4 c_\pm^2= \rop[v_0, \tilde H_1]$.  Calculate,
  \begin{align}
    \label{eq:61}
    |\tau|^{-2} \doph[v^\pm_{tip}, h^\pm_{tip} ]
    ={}&{} |\tau|^{-1}\left( \tilde H_0 - \rop[v_0, \tilde H_1]\right) \\
       &+ |\tau|^{-2} \left(
         v \frac{(\partial_\zeta\tilde H_{1})^2}{\tilde H}
         - \oh \zeta \partial_\zeta \tilde H_1
         + \tilde H_1
         + |\tau|^{-1} \tilde H_1
         \right)\\
       &+ |\tau|^{-2} \left[
         - \oh \zeta |\tau|^{-1}\partial_\zeta \tilde H_1 + \tilde H -  \alpha_p^2
         \right]
    &\\
    &= |\tau|^{-1}(\bar c_\pm - c _\pm)\\
    &+ |\tau|^{-2}G_2(\zeta) + |\tau|^{-3}G_3(\zeta)
  \end{align}
  The first line has the right sign because $\bar c_+ > c_0 > c_+$ and $\bar c_- < c_0 < c_-$.  $G_2$ and $G_3$ are bounded by some constant $M$ on $[0, \zeta_2]$.  Therefore if $|\tau|^{-1} < \frac{\bar c_+ - c_+}{2M}$ (and similarly for $c_-, \bar c_-$) the claim holds.
\end{proof}

\subsection{Gluing the parabolic and outer barriers}

\begin{lemma}
  Let $\epsilon$ and $\delta$ be given.  Consider the barriers for the outer and parabolic regions, given by Lemmas \ref{outer_barriers} and \ref{para_barriers}.  We can chose the constants $\hat a_\pm$ and $\hat b^\pm$ in the definition of the parabolic barriers, and take $\rho_2 = 2 \rho_1$, so that for all $\tau < \tau_*$
  \begin{align}
    v^+_{para}(\rho = \rho_1, \tau) < v^+_{out}(\rho_1, \tau), \quad v^+_{para}(\rho_2, \tau) > v^+_{out}(\rho_2, \tau) \\
    v^-_{para}(\rho_1, \tau) > v^-_{out}(\rho_1, \tau), \quad v^-_{para}(\rho_2, \tau) < v^-_{out}(\rho_2, \tau), 
  \end{align}
  and similarly for $h$.  Here we possibly decrease $\tau_*$ and increase $\rho_1$.

  In particular, we can choose $\hat a_+$ arbitrarily close to $a_+$, etc.
\end{lemma}

\begin{proof}
  We do the example of matching the supersolutions for $h$.  Write both of the supersolutions in the parabolic coordinates.
  \begin{align}
    \label{eq:25}
    \eta_{out}^+
    &= b_+^2\left( |\log \rho + \oh \tau| \right)\left( \rho^2 + d_+ \alpha_q^2 \right) \\
    &= \oh b_+^2 |\tau|  \left( \rho^2 + d_+ \alpha_q^2 \right) \left(|1 + 2 \tau^{-1} \log \rho | \right) \\
    \eta_{para}^+  &= \oh \hat b_+^2 |\tau|(\rho^2 + \alpha_q^2 (1 + D\tau^{-1}\rho^{-2}) ) \\
  \end{align}
  Write $\hat \delta = (b_+^2/ k^{-2})-1$ (note $\hat \delta \in (0, 1/2)$) and let
  \begin{align}
    \label{eq:44}
    M(\rho)
    &= \lim_{\tau \to -\infty}\frac{\eta_{out}^+}{\eta_{para}^+} \\
    &= \lim_{\tau \to -\infty}\frac{b_+^2}{\hat b_+^2} \frac{\rho^2 + d_+ \alpha_q^2}{\rho^2 + (1 + D\rho^{-2}|\tau|^{-1}) \alpha_q^2} \left(|1 + 2 \tau^{-1} \log \rho | \right) \\
    &= \frac{(1 + \delta)}{(1 + \hat \delta)}\frac{\rho^2 + d_+ \alpha_q^2}{\rho^2 + \alpha_q^2}
  \end{align}
  We will ensure that
  \begin{align}
    \label{eq:45}
    M(\rho_1) > 1 \text{ and } M(\rho_2) < 1.
  \end{align}
  and then the quotient $\eta^+_{out}/\eta^{+}_{para}$ will stay on the correct side of $1$ for small $\tau$ by continuity.
  
  Let $m_+ = (1 + \delta)/(1 + \hat \delta) = (b_+/\hat b_+)$. Since
  \begin{align}
    \label{eq:48}
    \lim_{\rho_2 \to \infty} M(\rho_2, \tau) = m_+
  \end{align}
  we will choose $m_+< 1$.  Now we want
  \begin{align}
    \label{eq:53}
    M(\rho_1) = \left( m_+ \right) \frac{\rho^2 + d_+ \alpha_q^2}{\rho^2 + \alpha_q^2} < 1
    < (m_+)\frac{(2\rho)^2 + d_+ \alpha_q^2}{(2\rho)^2 + \alpha_q^2}M(\rho_2)
  \end{align}
  From which we see we can first choose $\rho$ so large that $\frac{\rho^2 + d_+ \alpha_q^2}{\rho^2 + \alpha_q^2}$ is just barely larger than 1.  Then, choose $m_+ < 1$ small enough to bring $M(\rho_1)$ below 1, but not $M(\rho_2)$.
\end{proof}

\subsection{Gluing the tip and parabolic barriers}
\begin{lemma}
  Consider the definitions of the barriers in the parabolic and tip regions, in Lemmas \ref{para_barriers} and \ref{tip_barriers}.

  Let $\hat a_{\pm}$, $\hat b_{\pm}$, $\epsilon$, and $\rho_2$ be given.  We can choose $\zeta_1$ and $\zeta_2 = 2 \zeta_1$, as well as the constants $c_{\pm}, \bar c_{\pm}$, so that 
  \begin{align}
    v^+_{tip}(\zeta = \zeta_1, t) < v^+_{para}(\zeta_1, t), \quad v^+_{tip}(\zeta_2, t) > v^+_{para}(\zeta_2, t) \\
    v^-_{tip}(\zeta_1, t) > v^-_{para}(\zeta_1, t), \quad v^-_{tip}(\zeta_2, t) < v^-_{para}(\zeta_2, t), 
  \end{align}
  for all $t< t_*$, possibly making $t_*$ smaller.
\end{lemma}

\begin{proof}
Put the parabolic barriers in the tip coordinates.
\begin{align}
  \label{eq:62}
  v_{para}^{\pm} &= 2 \hat a_\pm^2 \left( |\tau|^{-1} +  \alpha_q^2 (1 \pm D\zeta^{-2}) \zeta^{-2} \right)\\
  H^{\pm}_{para} &= \tau^2 \left( \oh \hat b_{\pm}^2 \left(\alpha_q^2  (1 \pm D \zeta^{-2}) + |\tau|^{-1}\zeta^2 \right) \right) \\
  v_{tip}^\pm  &= \bry(c_\pm \zeta) + (1 \mp \epsilon) |\tau|^{-1} c_\pm^{-2} \cry(c_\pm \zeta) \\
  H_{tip}^\pm &= \tau^2 \left( \alpha_q^4 (\bar c_\pm)^2 + (1 \mp \epsilon)|\tau|^{-1} \ary(c_\pm \zeta) \right)
\end{align}

We only do the example of gluing the supersolutions for $H$.  The other cases are similar, although to glue the $v$ barriers we need to use the asymptotics for $\bry$.  The work for $v$ was also done in \cite{ACK}, Lemma 6.  Let
\begin{align}
  \label{eq:64}
  M(\zeta)
  &= \lim_{\tau \to -\infty}\frac{ H_{para}^+}{ H_{tip}^+}\\
  &= \frac{\hat b_+^2 }{(1 - \epsilon)(2\alpha_q^2\bar c_{+}^2)} (1 + D \zeta^{-2})
  = \frac{b_0^2 (\hat b_+^2/b_0^2) }{(1 - \epsilon)(2 \alpha_q^2 c_0^2 (\bar c_+^2/c_0^2))}(1 + D \zeta^{-2})\\
  &= \frac{(\hat b_+/b_0)^2}{(\bar c_+/c_0)^2 } \frac{(1 + D \zeta^{-2})}{(1-\epsilon)}
\end{align}
In the last calculation we used the definition of $c_0$.
We want this to be larger than 1 at $\zeta_1$ and smaller than 1 at $\zeta_2$.  Note that $M$ is decreasing and $m_+ \defeq \lim_{\zeta \to \infty}M(\zeta) = \frac{(\hat b_+/b_0)^2}{(\bar c_+/c_0)^2(1-\epsilon)}$, so we will certainly want to choose
\begin{align}
  \label{eq:65}
  m_+ < 1, \text{ equivalently } (\bar c_+/c_0)^2 > \frac{(\bar b_+/b_0)^2}{1-\epsilon},
\end{align}
but only barely.

If $\zeta_1 = \sqrt{D/\epsilon}$ as Lemma \ref{para_barriers} allows, then we can force then the desired inequalities read
\begin{align}
  \label{eq:46}
  (1 + \epsilon) m_+
  > 1
  > (1 + \epsilon (\zeta_1/\zeta_2)^2)m_+ = (1 + \epsilon/4) m_+
\end{align}
which we may make hold by choosing $m_+$ in $\left[\frac{1}{1 + \epsilon}, \frac{1}{1 + \epsilon/4} \right]$.
\end{proof}

\appendix
\section{Derivation of the $r$ coordinate system}
In this section we derive formulae for translating into the $r$ coordinate system.  Since $v = (\partial_s \phi)^2$ and $r=\phi$ we have
\begin{align*}
  \partial_s = \sqrt{v}\partial_r.
\end{align*}
From this we find,
\begin{align*}
  \partial_s^2
  &= v\partial_r^2 + \oh v^{-1/2}(\partial_s v)\partial_r \\
  &= v \partial_r^2 + \oh v_r\partial_r. 
\end{align*}
Using the chain rule and the evolution of $\phi$ \eqref{evo_phi} we have the following formula for comparing time derivatives with fixed $x$ or fixed $r$. 
\begin{align*}
  \partial_t \fixx
  &= \partial_t \fixr + (\partial_t\fixx \phi )\partial_r \\
  &= \partial_t \fixr + (\phi_{ss} + (p \psi^{-1}\psi_s + q \phi^{-1} \phi_s)\phi_s - \phi^{-1}\phi_s^2 - (q-1)\phi^{-1})\partial_r \\
  &= \partial_t \fixr + ((\sqrt{v})_s + (p \psi^{-1}\sqrt{v}\psi_r + q r^{-1} \sqrt{v})\sqrt{v} - r^{-1}v - (q-1)r^{-1})\partial_r \\
  &= \partial_t \fixr + \left(\sqrt{v}(\sqrt{v})_r  + p \psi^{-1}v\psi_r + (q-1) r^{-1}v  - (q-1)r^{-1}\right) \partial_r \\
  &= \partial_t \fixr + \left(\oh v_r  + p \psi^{-1}\psi_r v - (q-1) r^{-1}(1-v) \right) \partial_r \\
\end{align*}

Using these formulae we can convert the evolutions of $\phi$ and $\psi$ in the $s$ coordinate to the evolutions of $\phi$ and $v$ in the $r$ coordinate.  First, we derive the evolution for $\phi_s$.  Using the commutation formula for $\partial_t \fixx$ and $\partial_s$,
\begin{align*}
  \partial_t\fixx \phi_s
  &= \partial_s \left(\partial_t \fixx \phi \right) - \left( p \frac{\psi_{ss}}{\psi} + q \frac{\phi_{ss}}{\phi} \right)\phi_s\\
  &= \partial_s \left(
    \phi_{ss}
    +  p \frac{\psi_s}{\psi}\phi_s
    +(q-1)\phi^{-1}\left(\phi_{s}^2 - 1 \right)
  \right) \\
  &- \left( p \frac{\psi_{ss}}{\psi} + q \frac{\phi_{ss}}{\phi} \right)\phi_s\\
  &= \partial_s^2 \left(\phi_s \right)\\
  &+ p \frac{\psi_{ss}}{\psi} \phi_s - p \frac{\psi_{s}^2}{\psi^2}\phi_s + p \frac{\psi_s}{\psi}\phi_{ss}
    - (q-1)\phi^{-2}(\phi_s^2 - 1)\phi_s + 2 (q-1) \phi^{-1}\phi_s \phi_{ss}\\
  &- \left( p \frac{\psi_{ss}}{\psi} + q \frac{\phi_{ss}}{\phi} \right)\phi_s\\
  &= \partial_s^2 \left(\phi_s \right) + \left( p \frac{\psi_s}{\psi} + (q-2)\frac{\phi_s}{\phi}\right) \partial_s \phi_s\\
  &+ \left( -p \frac{\psi_s^2}{\psi^2} - (q-1) \phi^{-2}(\phi_s^2-1)\right)\phi_s
\end{align*}
From the evolution of $\phi_s$ we can derive the evolution of $v = \phi_s^2$.
\begin{align*}
  \partial_t \fixx v
  &= 2 \phi_s \partial_t (\phi_s) \\
  &= 2 \phi_s\partial_s^2 \left(\phi_s \right)
    + \left( p \frac{\psi_s}{\psi} + (q-2)\frac{\phi_s}{\phi}\right) 2\phi_s\partial_s \phi_s\\
  &+ \left( -p \frac{\psi_s^2}{\psi^2} - (q-1) \phi^{-2}(\phi_s^2-1)\right)2\phi_s\phi_s \\
  &= \partial_s^2 \left(\phi_s^2 \right) - \frac{1}{2\phi_s^2}(\partial_s\phi_s^2)^2
    + \left( p \frac{\psi_s}{\psi} + (q-2)\frac{\phi_s}{\phi}\right) \partial_s (\phi_s^2) \\
  &+ 2\left( -p \frac{\psi_s^2}{\psi^2} - (q-1) \phi^{-2}(\phi_s^2-1)\right)(\phi_s)^2 \\
  &= v_{ss} - \frac{1}{2v}v_s^2 + \left( p \frac{\psi_s}{\psi} + (q-2) \frac{\sqrt{v}}{r} \right) v_s \\
  &+ 2\left( -p \frac{\psi_s^2}{\psi^2} + (q-1) r^{-2}(1-v)\right)v \\
  &= v v_{rr} + \oh v_r^2 - \oh v_r^2 + \left( \sqrt{v}p \frac{\psi_r}{\psi} + (q-2) \frac{\sqrt{v}}{r} \right) \sqrt{v}v_r \\
  &+ 2\left( -p v\frac{\psi_r^2}{\psi^2} + (q-1) r^{-2}(1-v)\right)v \\
  &= v v_{rr} + (q-2) r^{-1}vv_r + 2 (q-1) r^{-2}v(1-v) + p \frac{\psi_r}{\psi}vv_r - 2 p v^2 \frac{\psi_r^2}{\psi^2}
\end{align*}
Now we switch from $\partial_t \fixx$ to $\partial_t \fixr$.
\begin{align*}
  \partial_t \fixr v
  &= \partial_t \fixx v - \left(\oh v_r  + p \psi^{-1}\psi_r v - (q-1) r^{-1}(1-v) \right) \partial_rv \\
  &= v v_{rr} + (q-2) r^{-1}vv_r + 2 (q-1) r^{-2}v(1-v) + p \frac{\psi_r}{\psi}vv_r - 2 p v^2 \frac{\psi_r^2}{\psi^2} \\
  &- \left(\oh v_r  + p \psi^{-1}\psi_r v - (q-1) r^{-1}(1-v) \right) v_r \\
  &= v v_{rr} - \oh v_r^2 -r^{-1}vv_r + 2 (q-1) r^{-2}v(1-v) + (q-1) r^{-1}v_r - 2 pv^2 \frac{\psi_r^2}{\psi^2} \\
  &= v v_{rr} - \oh v_r^2 -r^{-1}vv_r + 2 (q-1) r^{-2}v(1-v) + (q-1) r^{-1}v_r - \oh  pv^2 \frac{h_r^2}{h^2} \\
\end{align*}

Now, let's put the evolution for $\psi$ in terms of $r$.
\begin{align*}
  \partial_t \fixr \psi
  &= \partial_t \fixx \psi - \left(\oh v_r  + p \psi^{-1}\psi_r v - (q-1) r^{-1}(1-v) \right) \psi_r \\
  &= \psi_{ss}
    + \left( p \frac{\psi_s}{\psi} + q \frac{\phi_s}{\phi} \right)\psi_s
    - \psi^{-1}\psi_{s}^2 - (p-1)\psi^{-1} \\
  &- \left(\oh v_r  + p \psi^{-1}\psi_r v - (q-1) r^{-1}(1-v) \right) \psi_r \\
  &= v\psi_{rr} + \oh v\psi_r^2
    + \left( p \frac{\psi_r}{\psi} + q \frac{1}{r} \right) v\psi_r
    - \psi^{-1}v\psi_r^2 - (p-1)\psi^{-1} \\
  &- \left(\oh v_r  + p \psi^{-1}\psi_r v - (q-1) r^{-1}(1-v) \right) \psi_r \\
  &= v\psi_{rr} + \oh v\psi_r^2
    -\psi^{-1}v\psi_r^2 - (p-1)\psi^{-1} \\
  &- \left(\oh v_r  - (q-1) r^{-1} -r^{-1}v\right) \psi_r \\
  &= v\psi_{rr} + \oh v\psi_r^2 - \oh v_r \psi_r + \frac{q-1+v}{r}\psi_r - v \frac{\psi_r^2}{\psi} - (p-1)\frac{1}{\psi}
\end{align*}
Note the cancellation of the $\left( p\frac{\psi_r}{\psi} + q \frac{\phi_r}{\phi} \right)$ occurs because this term comes from a
drift in $x$ which is removed by considering a geometric coordinate $r$.

Finally, derive the evolution of $h = \psi^2$.
\begin{align*}
  \partial_t \fixr h
  &= v h_{rr} - v \frac{h_r^2}{h} + r^{-1}(q-1) h_r + r^{-1}vh_r - (p-1)
\end{align*}

\section{Formal asymptotics before the singularity}\label{before_asymptotics}
In this section we formally derive the asymptotics of a flow into a singular metric of the form assumed in Theorem \ref{main_theorem}.  We work in the $s$ coordinate.  Recall that under Ricci flow,
\begin{align*}
  \partial_t \fixx \psi &= \psi_{ss}
                          + \left( p \frac{\psi_s}{\psi} + q \frac{\phi_s}{\phi} \right)\psi_s
                          - \psi^{-1}\psi_{s}^2 - (p-1)\psi^{-1}, \\
  \partial_t \fixx \phi &= \phi_{ss}
                          + \left( p \frac{\psi_s}{\psi} + q \frac{\phi_s}{\phi} \right)\phi_s
                          - \phi^{-1}\phi_{s}^2 - (q-1)\phi^{-1}, \\
  \partial_t \fixx \log s' &= p\frac{\psi_{ss}}{\psi} + q\frac{\phi_{ss}}{\phi}
\end{align*}
To convert the time derivatives to the $s$ coordinate we may use that for any evolving function $f$,
\begin{align*}
  \partial_t \fixs f
  &= \partial_t \fixx f - \left[ \partial_t \fixx s \right]  \partial_s f  \\
%  &= \partial_t \fixx f - \left[ \int_0^x \partial_t \fixx q dx \right] \partial_s f \\
  &= \partial_t \fixx f - \left[ \int_0^s \frac{\partial_t \fixx s'}{s'} ds \right] \partial_s f \\
%  &= \partial_t \fixx f - \left[ \int_0^s p\frac{\psi_{ss}}{\psi} + q \frac{\phi_{ss}}{\phi} ds \right] \partial_s f \\
  &= \partial_t \fixx f - I[\psi, \phi]\partial_s f
\end{align*}
Where in the last line we have named $I[\psi, \phi] = \int_0^s p \frac{\psi_{ss}}{\psi} + q \frac{\phi_{ss}}{\phi} ds$.  Using this we find,
\begin{align}
  \partial_t \fixs \psi &= \psi_{ss}
                          + \left( p \frac{\psi_s}{\psi} + q \frac{\phi_s}{\phi} - I[\psi, \phi]\right)\psi_s
                          - \psi^{-1}\psi_{s}^2 - (p-1)\psi^{-1}
  \\
  \partial_t \fixs \phi &= \phi_{ss}
                          + \left( p \frac{\psi_s}{\psi} + q \frac{\phi_s}{\phi} - I[\psi, \phi]\right)\phi_s
                          - \phi^{-1}\phi_{s}^2 - (q-1)\phi^{-1}
\end{align}

Before the singular time, these functions will have the boundary conditions at $s=0$:
\begin{align}
  \label{eq:14}
  \phi > 0, &\quad \psi = 0,\\
  \partial_s \phi = 0, &\quad \partial_s \psi = 1.
\end{align}
Given this, we rewrite $\psi = s (1 + \tilde \psi)$ where now $\psi$ will have $\tilde \psi_s = 0$ at $s = 0$, if the metric is smooth.  We may integrate $I[\phi, s(1 + \tilde \psi)]$ by parts to find
\begin{align*}
  \left[ \left( p \frac{\psi_s}{\psi} + q \frac{\phi_s}{\phi} \right) - I[\psi, \phi] \right]
%  &= \left[ \left( p \left(\frac{1}{s} + \frac{\tilde \psi_s}{1 + \tilde \psi} \right) + q \frac{\phi_s}{\phi} \right) - I[\psi, \phi] \right] \\
  &= \left[
    \frac{p}{s}
    - \int_0^s \left(
    q \frac{\phi_s^2}{\phi^2} + p \left(  \frac{\tilde \psi_s^2}{ (1 + \tilde \psi)^2 } + \frac{2 \tilde \psi_s}{s (1 + \tilde \psi)} \right)
    \right)ds  \right].
\end{align*}
One may then compute,
\begin{align}
  \partial_t \fixs \tilde \psi 
  &=
  \tilde \psi_{ss}
  + p \frac{\tilde \psi_s}{s}
    - 2p \frac{1}{s} \int_0^s \frac{\tilde \psi_s}{s} ds
    + 2(p-1)\frac{\tilde \psi}{s^2} \label{tpsicomp_lin}\\
  &- 2p \frac{\tilde \psi}{s} \int_0^s \frac{\tilde \psi_s}{s(1+ \tilde \psi)} ds
    - 2p \frac{1}{s} \int_0^s \frac{\tilde \psi_s}{s} \left(    \frac{1}{(1+ \tilde \psi)} - 1 \right) ds
      + \frac{(p-1)}{s^2}\left(  \left(1 + \tilde \psi \right) - \frac{1}{1 + \tilde \psi} - 2 \tilde \psi\right) \label{tpsicomp_linrest}\\
  &- \frac{1 + \tilde \psi}{s} \int_0^s \left( q \frac{\phi_s^2}{\phi^2} + p \frac{\tilde \psi_s^2}{(1 + \tilde \psi)^2} \right)ds
    - \tilde \psi_s  \int_0^s \left( q \frac{\phi_s^2}{\phi^2} + p \left(\frac{\tilde \psi_s^2}{(1 + \tilde \psi)^2} + \frac{2 \tilde \psi_s}{s(1+\tilde \psi)}\right) \right)ds   - \frac{\tilde \psi_s^2}{s} \label{tpsicomp_quad},\\
\end{align}
\begin{align}
  \partial_t \fixs \phi
  &= \phi_{ss} + \frac{p}{s}\phi_s  - (q-1)\phi^{-1}\\
  &- \int_0^s
    \left(
    q \frac{\phi_s^2}{\phi^2}
    + p \left(  \frac{\tilde \psi_s^2}{ (1 + \tilde \psi)^2 } +
 \frac{2 \tilde \psi_s}{s (1 + \tilde \psi)} \right)
    \right) \phi_s
    - \frac{\phi_s^2}{\phi}.
\end{align}
Here we have organized the equations so that the first lines are linear.

\subsection{Overview of Formal Asymptotics}
Our inspection of the shape of the Ricci flow starts with what we are most confident in.  Our primary assumption is that the flow develops a Type-I singularity modeled on $\real^{p+1}\times S^q$.  This means that under a rescaled flow, the metric approaches the standard metric on $\real^{p+1} \times S^q$, which is a fixed point for the flow.

In order to study the flow more closely, we expand the solution around the fixed point.  We assume that the solution approaches the fixed point at the same rate as in previously studied cases (The case $p=1$), which gives us an asymptotic expansion.  Our goal is to be able to use this asymptotic expansion to learn something about the ``naked-eye'' final time profile by looking at this asymptotic expansion.  (Note however that the ``neck'' region where this asymptotic expansion is valid becomes a single point at the singular time.)

The asymptotic expansion around the fixed point is actually not enough to tell us about the naked-eye profile.  It is, however, enough to tell us about a region farther from the neck, which still disappears at the singular time.  Then, information from this region is enough to tell us about the naked-eye profile.

\subsection{The neck region}
Our primary assumption is as follows:
\begin{assumption}\label{assumption_main}
  On the submanifold $\{s = 0\}$, and at time $T$, the metric has a type-I singularity modeled on $\real^{p+1} \times S^q$.  Precisely,
  \begin{itemize}
  \item At $s = 0$, $|\Rm| = O\left( \frac{1}{T-t} \right)$
  \item The metrics
      $G(t) = \frac{1}{T-t}(X(t))^*g(t)$
    converge to the soliton metric $G_{sol}$ on $\real^{p+1} \times S^q$, in compact neighborhoods of the submanifold $\{s = 0\}$.  Here $X(t)$ is a family of diffeomorphisms which integrate the soliton vector field.
  \end{itemize}
\end{assumption}
In particular, Assumption \ref{assumption_main} implies the following on the level of the functions $s, \phi, \psi$.  Set:
\begin{align*}
  \sigma = (T-t)^{-1/2}s \quad \Phi = (T-t)^{-1/2}\phi \quad \Psi = (T-t)^{-1/2}\psi \quad \tilde \Psi = \psi
\end{align*}
Then $\Phi \to \sigma$ and $\Psi \to \sqrt{2(q-1)}$ in regions $\{\sigma < A \}$.  We will call such a region $\{ \sigma < A \}$ a neck region.

Notice that because $s$ is a geometric coordinate for the metrics $g$ (and $\sigma$ is a geometric coordinate for $G$) we do not have to worry about the diffeomorphisms $X(t)$.  The soliton metric is given by $\tilde \Psi = 0$, $\Phi = \alpha_q \defeq \sqrt{2(q-1)}$, and is a fixed point of the rescaled system.  Write $\Phi = \alpha (1 + \tilde \Phi)$ so that now $\tilde \Psi = 0, \tilde \Phi = 0$ is a fixed point.  In full, the evolution of $\tilde \Phi, \tilde \Psi$ is
\begin{align}
  \partial_\tau \fixsig \tildehack \Psi
  &= \tildehack \Psi_{\sigma \sigma}
    + p \frac{1}{\sigma} \tildehack\Psi_\sigma
    - \oh \sigma \tildehack \Psi_\sigma
    + \frac{2(p-1)}{\sigma^2}\tilde \Psi
    - \frac{2p}{\sigma}\left[ \int_0^\sigma \frac{\tilde \Psi_\sigma}{\sigma} \right] \\
  &- 2p \frac{\tilde \Psi}{s} \int_0^s \frac{\tilde \Psi_s}{s(1+ \tilde \Psi)} ds
    - 2p \frac{1}{s} \int_0^s \frac{\tilde \Psi_s}{s} \left(    \frac{1}{(1+ \tilde \Psi)} - 1 \right) ds
      + \frac{(p-1)}{s^2}\left(  \left(1 + \tilde \Psi \right) - \frac{1}{1 + \tilde \Psi} - 2 \tilde \Psi\right) \\
  &- \frac{1 + \tilde \Psi}{s} \int_0^s \left( q \frac{\tilde \Phi_s^2}{(1 + \tilde \Phi)^2} + p \frac{\tilde \Psi_s^2}{(1 + \tilde \Psi)^2} \right)ds
    - \tilde \Psi_s  \int_0^s \left( q \frac{\tilde \Phi_s^2}{\tilde \Phi^2} + p \left(\frac{\tilde \Psi_s^2}{(1 + \tilde \Psi)^2} + \frac{2 \tilde \Psi_s}{s(1+\tilde \Psi)}\right) \right)ds   - \frac{\tilde \Psi_s^2}{s} \label{evo_tildepsi_full}\\
\end{align}
\begin{align}
  \partial_\tau \fixsig \tildehack \Phi
  &= \tildehack \Phi_{\sigma\sigma}
    + p \frac{1}{\sigma} \tildehack\Phi_\sigma - \oh \sigma \tildehack\Phi_\sigma + \tildehack \Phi \\
  &+ \oh \left( 1+\tildehack\Phi - \frac{1}{1 + \tildehack \Phi} - 2\tildehack \Phi \right) \\
  &+
    \left[
    \int_0^\sigma q \frac{\tildehack\Phi_\sigma^2}{(1 + \tildehack\Phi)^2} + \frac{p\tilde \Psi_\sigma^2}{(1 + \tilde \Psi)^2 } + \frac{2p\tilde \Psi_\sigma}{\sigma ( 1 + \tilde \Psi)}
    \right]\tildehack\Phi_\sigma
  - \left[ \frac{\Phi_\sigma}{1 + \tildehack\Phi} \right]\tildehack\Phi_\sigma \label{evo_tildephi_full}
\end{align}
Here, the first lines in both evolutions are linear and the others are at least quadratic in $\tilde \Phi$ and $\tilde \Psi$.

We proceed to study these linearizations.  Most familiar is the linearization of the evolution for $\tilde \Phi$.

\subsubsection{The linearization for $\tilde \Phi$.}
We study the operator
\begin{align}
  \partial_\sigma^2 + \frac{p}{\sigma}\partial_\sigma - \oh \sigma \partial_\sigma + 1. \label{philinearization_sigma}
\end{align}
For smoothness of the metric, $\tilde \Phi$ as a function of $\sigma$ must extend to an even function around zero.  On functions with this property the operator \eqref{philinearization_sigma} can be recognized as the operator
\begin{equation*}
 L_\Phi =  \lap_{\real^{p+1}} - \oh \vec{x} \cdot \nabla  + 1
\end{equation*}
acting on a rotationally symmetric function in $\real^{p+1}$.  This operator is self-adjoint on
\begin{align*}
L^2(\real_+, \sigma^{p}e^{\sigma^2/4}) = \text{ rotationally symmetric functions of }L^2(\real^{p+1}, e^{|-\vec{x}|^2/4}).  
\end{align*}
  The eigenvalues of $\lap_{\real^{p+1}} - \oh \vec{x} \cdot \nabla$ in $L^2(\real^{p+1}, e^{-|\vec{x}|^2/4})$ are the $(p+1)$-dimensional Hermite polynomials, which are all given by products of one-dimensional hermite polynomials in each coordinate.  The eigenvalues of $\partial_\sigma^2 + \frac{p}{\sigma}\partial_\sigma - \oh \sigma \partial_\sigma$ come from those hermite polynomials which happen to be rotationally symmetric.  The eigenspaces of $\partial_\sigma^2 + \frac{p}{\sigma}\partial_\sigma - \oh \sigma \partial_\sigma + 1$ (including the $+1$ term which shifts eigenvalues) are as follows:
\begin{itemize}
\item Constants are eigenfunctions with eigenvalue $1$.
\item There is a one-dimensional nullspace.  $f(x_1) = x_1^2 - 2$ is a one-dimensional hermite polynomial, and if $x_1, x_2, \dots, x_{p+1}$ are the coordinates of $\real^{p+1}$ then
  \begin{align*}
    (x_1^2-2) + (x_2^2 -2) + \cdots + (x_{p+1}^2 - 2) = (x_1^2 + x_2^2 + \cdots + x_{p+1}^2) - 2(p+1)
  \end{align*}
  is in the nullspace of $\lap_{\real^{p+1}} - \oh x \cdot \nabla + 1$, and
  \begin{align*}
    \sigma^2 - 2(p+1)
  \end{align*}
  is in the nullspace of
  \begin{align*}
    \partial_\sigma^2 + \frac{p}{\sigma}\partial_\sigma - \oh \sigma \partial_s + 1
  \end{align*}
\item All further eigenspaces are negative.
\end{itemize}

Remember that we are considering a flow in which $\Phi$ approaches $\alpha_q$, i.e. $\tilde \Phi$ approaches zero.  Therefore the constant component of $\tilde \Phi$ should get smaller, and in the $\tau \to \infty$ limit should disappear.  Therefore we expect the nullspace to play the biggest role; this was the case in \cite{AKPrecise} as well.  Lemma \ref{lemma:philin_nullspace} studies the nullspace more explicitly.

\begin{lemma}\label{lemma:philin_nullspace}
  In any neighborhood of zero, the only solutions to
  \begin{align}
    \partial_\sigma^2f + \left(\frac{p}{\sigma} - \frac{\sigma}{2} \right)\partial_\sigma f + 1 = 0 \label{philin_nullspace}
  \end{align}
  which are $C^2$ at zero are multiples of
  \begin{align*}
    f(\sigma) = (\sigma^2 - 2(p+1))
  \end{align*}
  where $k$ is arbitrary.
\end{lemma}
\begin{proof}
  \eqref{philin_nullspace} is a degree two linear ODE with an isolated regular singular point at $\sigma = 0$.  
  The method of Frobenius yields two independent solutions to \eqref{philin_nullspace} around zero: one is $\sigma^2 -2(p+1)$ and the other blows up at $\sigma = 0$ with order $O(\sigma^{-(p+1)})$.
\end{proof}

\subsubsection{The linearization for $\tilde \Psi$}
The linearization of the evolution for $\tilde \Psi$ is
\begin{align*}
  L_\Psi[f] = \partial_\sigma^2 f + \frac{p}{\sigma}\partial_\sigma f - \frac{\sigma}{2}\partial_\sigma f
  + \frac{2(p-1)}{\sigma^2}f - \frac{2p}{\sigma} \left[ \int_0^\sigma \frac{f_{\sigma}}{\sigma} d\sigma\right].
\end{align*}
This is more complicated because there is a nonlocal term in the linearization.  We can begin by computing it on monomials.  Since $\Psi$ must be even and vanish at zero, we just compute $L_\Psi[f]$ for $f = \sigma^{2k}$, $k\geq 1$.
\begin{align*}
  L_\Psi[\sigma^{2k}]
  &= 2k(2k-1)\sigma^{2k-2} + p2k\sigma^{2k-2} - \frac{1}{2}2k\sigma^{2k} + 2(p-1)\sigma^{2k-2}
    - \frac{2p}{\sigma}\int_0^\sigma 2k\sigma^{2k-2} \\
  &= \left(2k(2k-1) + 2pk + 2(p-1) - \frac{4pk}{2k-1}\right)\sigma^{2k-2} - k\sigma^{2k}
\end{align*}
The pleasant part of the situation is that the operator acts on the monomials as an upper-triangular matrix.
    %     An important sanity check is that for $k=1$ the coefficient of $\sigma^{2k-2} = 1$ term vanishes; this happens because Ricci flow keeps the point $\sigma = 0$ smooth.
Therefore one can read off the eigenvalues in the span of the monomials.   They are the coefficients of $\sigma^{2k}$ above, that is $-k$ for $k \in \nats , k \geq 1$.  This method also gives a formula for computing the eigenfunctions.

One can also study the functions satisfying $L_\Psi[f]=\lambda f$ by multiplying by $\sigma$ and differentiating with respect to $\sigma$, arriving at
\begin{align}
  \sigma\partial_\sigma^3 f
  + \left( (p+2) - \frac{\sigma^2}{2} \right)\partial_\sigma^2 f
  + \left( \frac{(4p-2)}{\sigma} - (1 + \lambda)\sigma \right) \partial_\sigma f
  - \left( \frac{2(p-1)}{\sigma^2} - \lambda \right)f
  = 0
\end{align}

In any case we find the following.
\begin{lemma}
  \label{lemma:psilin_nullspace}
  The operator $L_\Psi$ has a strictly negative spectrum in $L^2(\real_+, \sigma^pe^{-\sigma^2/4})$.
\end{lemma}

\subsubsection{The Ansantz for $\tilde \Phi, \tilde \Psi$}
Our assumption that the rescaled metric approaches the soliton $\real^{p+1} \times S^q$ says that $\tilde \Phi$ and $\tilde \Psi$ both approach zero as $\tau \to \infty$.  We make two further assumptions about the rate of this convergence.
\begin{assumption}\label{assume:parabolic_tauinv_convergence}
  The limits
  \begin{align}
    \label{eq:2}
    \tilde \Psi_1(\sigma) \defeq \lim_{\tau \to \infty}\tau \tilde \Psi(\sigma, \tau),
    \quad
    \tilde \Phi_1(\sigma) \defeq \lim_{\tau \to \infty}\tau \tilde \Phi(\sigma, \tau)
  \end{align}
  both exist.  The convergence happens in $C^2$ on any region $\{\sigma < \sigma_\# \}$. 
\end{assumption}
\begin{assumption}\label{assume:parabolic_weightedspace}
  As functions of $\sigma$, $\tilde \Psi_1$ and $\tilde \Phi_1$ are in $L^2(\real_+, \sigma^pe^{-\sigma^2/4})$.
\end{assumption}
By Assumption \ref{assume:parabolic_tauinv_convergence} we can write
\begin{align}
  \label{eq:3}
  \tilde \Psi = \tau^{-1}(\tildehack\Psi_1(\sigma) + \tildehack\Psi_1^{(err)}(\sigma, \tau)),
  \quad
  \tilde \Phi = \tau^{-1}(\tildehack\Phi_1(\sigma) + \tildehack\Phi_1^{(err)}(\sigma, \tau))
\end{align}
where $\Psi_1^{(err)}$ and $\Phi_1^{(err)}$ converge $C^2_{loc}$ to zero as $\tau \to \infty$.  Plugging this into the evolution equations and bounding nonlinear terms shows $L_\Phi[\tilde \Phi_1] = L_{\Psi}[\tilde \Psi_1] = 0$.  Then Lemmas \ref{lemma:philin_nullspace} and \ref{lemma:psilin_nullspace} with Assumption \ref{assume:parabolic_weightedspace} show that for some $k_0$
\begin{align}
  \label{eq:4}
  \tildehack\Psi_1 = 0, \quad \tildehack\Phi_1 = k_0(\sigma^2 - 2(p+1)).
\end{align}
\begin{remark}
  In \cite{AKPrecise}, the authors rigorously found the value of $k_0$ for the case $p=0$.  Formally one can find the value by analyzing the evolution of the inner product
  \begin{align}
    \label{eq:17}
    k_0(\tau) = \int_\real \left(\tilde \Phi(\sigma, \tau) \right) \left(\sigma^2-2(p+1) \right) \left( \sigma^pe^{-\sigma^2/4} \right) d\sigma ,
  \end{align}
  taking into account quadratic terms in the evolution of $\tilde \Phi$.
\end{remark}

\subsubsection{Validity}
We come back to
\begin{align}
  \Psi(\sigma, \tau) &= \sigma \left(  1 + \tau^{-1}\tilde\Psi_1^{(err)}(\sigma, \tau) \right), \label{guess_psi}\\
  \Phi(\sigma, \tau) &= \alpha_q\left( 1 + \tau^{-1}k(\sigma^2 - 2(p+1)) + \tau^{-1}\tilde \Phi_1^{(err)}(\sigma, \tau) \right), \label{guess_phi}
\end{align}
and study the regions where it is valid for $\Psi_1^{(err)}$ and $\Phi_1^{(err)}$ to be small.  Using \eqref{guess_psi}, \eqref{guess_phi} we can derive evolution equations for $\Psi_1^{(err)}$ and $\Phi_1^{(err)}$.  These will be parabolic equations with a source term which is $O((\tau^{-1}\sigma^2)^2, (\tau^{-1}\sigma^2) \to 0)$.

Therefore it is consistent to assume that
\begin{align}
  \label{eq:7}
  \tilde \Psi_1^{(err)}(\sigma, \tau),\, \tilde \Phi_1^{(err)}(\sigma, \tau) = o(1; \tau \to \infty)
  \text{ if } \sigma = o(\sqrt{\tau}, \tau \to \infty).
\end{align}
We study regions where $\sigma = O(\sqrt{\tau})$ in the next section.

% \subsubsection{Postconclusion: uselessness for $s > 0$.}
% The limit of $\Phi_{neck}$ as $t \to T$ is (as prescribed) $\alpha_q$.

% Write $\phi_{neck}$ in terms of $\sigma$:
% \begin{align*}
%   \phi_{neck} &= \sqrt{T-t}\alpha_q + k \frac{\sqrt{T-t}}{|\log (T-t)|} (\sigma^2 - 2(p+1))
% \end{align*}
% For fixed $\sigma_\#$ for what $t_\#$ is $\sigma(s, t) = \sigma_\#$?  It's when $\sqrt{T-t_\#} = \frac{s}{\sigma_\#}$.
% \begin{align*}
%   \phi_{neck}(\sigma = \sigma_\#, t = t_\#)
%   &= \alpha_qs\frac{1}{\sigma_\#}
%     + k s
%     \frac{1}{|2 \log s - 2 \log \sigma_\#|}
%     \left( \sigma_\# - 2(p+1)\frac{1}{\sigma_\#} \right)
% \end{align*}

\subsection{The Intermediate Region}

When $\sigma \sim \sqrt{\tau}$, the assumption that the error term for the neck approximation $\hat \Phi$ is small is no longer feasible.  Let us introduce scaled functions
\begin{align*}
  \xi = \sigma/\sqrt{\tau} \qquad Y = \Psi/\sqrt{\tau} \qquad \tildehack Y = \tildehack \Psi
  \qquad \text{ (so } Y = \xi (1 + \tilde Y)\text{ )}
\end{align*}
Then
\begin{align*}
  \partial_\tau \fixsig
  &= \partial_\tau \fixxi + (\partial_\tau \xi) \partial_\xi \\
  &= \partial_\tau \fixxi - \oh \tau^{-1}\xi \partial_\xi \\
  \partial_\sigma
  &= (\partial_\sigma \xi) \partial_\xi \\
  &= \tau^{-1/2}\partial_\xi
\end{align*}
Calculate the evolutions.  Every term in the right hand side of $\partial_\tau \fixsig \Psi$ scales to have a $\tau^{-1}$ coefficient except for the $-\oh \sigma \Psi_\sigma$ which just scales to $- \oh \xi \Psi_\xi$.  The evolution of $\Phi$ also has reaction terms with no $\tau^{-1}$ coefficient.
\begin{align}
   \partial_\tau \fixxi \tilde Y
  &= - \oh \xi \tildehack Y_\xi \label{evo_psi_int}\\
  &+\tau^{-1}
    \left(
    + \oh \xi Y_\xi
    + \tilde Y_{\xi\xi}
    + p\frac{1}{\xi}\tilde Y_\xi
    + \frac{(p-1)}{\xi^2}\left( 1+\tilde Y-\frac{1}{1+\tilde Y}\right)
    + \frac{2p}{\xi} \left[ \int_0^\xi \frac{\tilde Y_\xi}{\xi(1+\tilde Y)} \right]
  \right)\nonumber\\
  &+ \tau^{-1}
    \left(
    \frac{\tilde Y_\xi^2}{1 + \tilde Y} 
    +
    \frac{1}{\xi}
    \left[
    \int_0^\xi q \frac{\Phi_\xi^2}{\Phi^2}
    + \frac{p\tilde Y_\xi^2}{(1 + \tilde Y)^2 }
    \right]
    (1 + \tilde Y + \xi \tilde Y_\xi)
    +
    \frac{1}{\xi}
    \left[
    \int_0^\xi\frac{2p\tilde Y_\xi}{\xi ( 1 + \tilde Y)}
    \right]
    (\tilde Y + \xi \tilde Y_\xi)
    \right)\nonumber\\
  &\nonumber\\
  \partial_\tau \fixxi \Phi
  &= - \oh \xi \Phi_\xi - (q-1)\Phi^{-1} + \oh \Phi \label{evo_phi_int}\\
  &+\tau^{-1}\left(
    + \oh \xi \Phi_\xi + \Phi_{\xi \xi} + p \frac{1}{\xi}\Phi_\xi
    \right)
  \nonumber\\
  &+
    \tau^{-1}\left(
    \left[
    \int_0^\xi q \frac{\Phi_\xi^2}{\Phi^2} + \frac{p\tilde Y_\xi^2}{(1 + \tilde Y)^2 } + \frac{2p\tilde Y_\xi}{\xi ( 1 + \tilde Y)}
    \right]\Phi_\xi
    - \left[ \frac{\Phi_\xi}{\Phi} \right]\Phi_\xi
    \right)\nonumber
\end{align}

We assume that $\Phi$ and $\tilde Y$ have limits as $\tau \to \infty$:
\begin{assumption} \label{assume:intermediate_firstorder}
  The limits
  \begin{align}
    \label{eq:8}
    \Phi_{int0}(\xi) = \lim_{\tau \to \infty} \Phi(\tau, \xi)
    \quad \text{ and } \quad
    \tilde Y_{int0}(\xi) = \lim_{\tau \to \infty} \tilde Y(\tau, \xi)
  \end{align}
  exist.  The limit occurs in $C^2$ on regions $\{\xi < \xi_\#\}$.
\end{assumption}
Under Assumption \ref{assume:intermediate_firstorder}, $\Phi_{int0}$ and $\tilde Y_{int0}$ will solve 
\begin{align*}
  0 &= - \oh \xi \tildehack Y_{int0, \xi} \\
  0 &= - \oh \xi \Phi_{int0,\xi} - (q-1) \Phi_{int0}^{-1} + \oh \Phi_{int0}
\end{align*}
so they are 
\begin{align*}
  \tilde Y_{int0} &= k_0 \\
  \tilde \Phi_{int0} &= \sqrt{k_1 \xi^2 + \alpha_q^2}
\end{align*}
(Recall $\alpha_q = \sqrt{2(q-1)}$).

\subsubsection{Matching}
We want the intermediate approximations to be valid at the boundary of the neck region, and match the neck approximations.  Therefore, let us say we hope the intermediate solutions to be valid on regions of the form
$$\{\sigma \geq \sigma_0 \text{ and } \xi < \xi_0 \} = \{\sqrt{\tau} \xi \geq \sigma_0 \text{ and } \xi < \xi_0 \} = \left\{\frac{\sigma_0}{\sqrt{\tau}} \leq \xi \leq \xi_0\right\}$$
Set
\begin{align*}
  \tilde Y_{int} &= k_0 \\
  \Phi_{int} &= \sqrt{k_1 \xi^2 + \alpha_q^2}
\end{align*}
with $k_0$ and $k_1$ to be determined.

First, unravel the definition of $\tilde Y$.
\begin{align*}
  Y_{int} &= \xi (1 + \tilde Y_{int})  = \xi (1 + k_0)\\
  \Psi_{int} &= \sqrt{\tau}Y_{int} = \sigma(1 + k_0)
\end{align*}
Matching $\Psi_{int}$ with $\Psi_{neck} = \sigma$ gives $k_0 = 0$.

Putting $\Phi_{neck}$ in terms of $\xi$ gives
\begin{align*}
  \Phi_{neck} = \alpha_q(1 + k(\xi^2 - 2(p+1)\tau^{-1}))  
\end{align*}
so when $\xi$ is small and $\tau$ is large
\begin{align}
  \label{eq:10}
  \Phi_{neck} \approx \alpha_q, \quad \Phi_{neck, \xi} \approx 0, \quad \Phi_{neck, \xi \xi} \approx 2k\alpha_q, \\
  \Phi_{int} \approx \alpha_q, \quad \Phi_{int, \xi} \approx 0, \quad \Phi_{int, \xi \xi} \approx \frac{k_1}{\alpha_q}.
\end{align}
So we choose $k_1 = 2k\alpha_q^2$ and have
\begin{align*}
  \Phi_{int} = \alpha_q \sqrt{2k \xi^2 + 1}
\end{align*}

\subsection{Outer region}
Both the neck and intermediate regions shrink to the singular submanifold at $t = T$.  Now we attempt to use the intermediate approximations to get information about the solution at time $T$, outside of the singular submanifold.

From our considerations in the intermediate region, with Assumption \ref{assume:intermediate_firstorder} and the matching we can write 
\begin{align}
  \psi &= \left( (T-t)^{1/2}|\log(T-t)|^{1/2} \right) \xi \left( 1 + \tilde Y_{int0}^{(err)} \right) = s\left( 1 + \tilde Y_{int0}^{(err)} \right), \label{psi_from_int} \\
  \phi &=   \sqrt{(T-t)}\alpha_q\left( \sqrt{2k\xi^2 + 1} + \Phi_{int0}^{(err)}\right)
         = e^{-\tau/2}\alpha_q\left( \sqrt{2k\xi^2 + 1} + \Phi_{int0}^{(err)}\right), \label{phi_from_int}
\end{align}
where the error terms $Y_{int0}^{(err)}$ and $\Psi_{int0}^{(err)}$ are $o(1; \tau \to \infty)$ on sets $\{ \xi < \xi_\# \}$.  Recall that $\tau, \xi, s, t$ are related by
\begin{align}
  (T-t)\tau = e^{-\tau}\tau = \frac{s^2}{\xi^2}. \label{xi_s_tau_relation}
\end{align}

We will take $\xi_\#$ to infinity and $s$ to 0, but still have the error terms go to zero.  Let $\tau_\#(\xi_\#)$ be large enough so that as $\xi_\# \to \infty$
\begin{align}
  \label{eq:11}
  Y_{int0}^{(err)}(\xi_\#, \tau=\tau_\#(\xi_\#)) \to 0, \quad   \Phi_{int0}^{(err)}(\xi_\#, \tau=\tau_\#(\xi_\#)) \to 0.
\end{align}
Because the error terms are continuous, it is possible to make $\tau_\#(\xi_\#)$ continuous and increasing.  As a further requirement on $\tau_\#(\xi_\#)$ we ask that
\begin{align}
  \label{eq:16}
  \frac{\log(\xi_\#)}{\log(s_\#)}
  = \frac{ \log \xi_\# }{\log \left(  \sqrt{e^{-\tau_\#}\tau_\# }\xi_\# \right) }
  \to 0
  ,\qquad
  \text{equivalently } \frac{\log \xi_\#}{\tau_\#} \to 0.
\end{align}
Now consider $\xi_\#$, $t_\#$, and $\tau_\#$ as functions of $s_\#$, satisfying
\begin{align}
  e^{-\tau_\#}\tau_\# = \frac{s^2_\# }{\xi_\#^2}. \label{tausharp_xisharp} 
\end{align}
Find an expression for $(T-t_\#) = e^{-\tau_\#}$ by taking the logarithm of both sides of \eqref{tausharp_xisharp} and then dividing $e^{-\tau_\#}\tau_\#$ by $\tau_\#$:
\begin{align}
  \label{eq:15}
  -\tau_\# + \log \tau_\# = -2 \log s_0 - 2 \log \xi_\#
  &\implies \tau_\#(1 + o(1)) = -2 \log s_\#\left(1 - \frac{\log \xi_\#}{\log s_\#} \right) \\
  &\implies \tau_\# = \left(2 \log s_\# \right) \left( 1 + o(1) \right) \\
  &\implies e^{-\tau_\#} = \frac{s^2_\#}{\xi_\#^2} \frac{1}{-2 \log s_\#}\left(1 + o(1) \right)
\end{align}

Now, evaluate \eqref{psi_from_int}, \eqref{phi_from_int} at $s = s_\#, t=t_\#$.
\begin{align}
  \psi(s, t_\#)
  &= s\left( 1 + \tilde Y^{(err)}_{int0}\left(  \xi_\#, t_\#  \right) \right) \\
  &= s\left( 1 + o(1; t_\# \to \infty) \right)\label{psi_sharp_evaluate}
\end{align}
% \begin{align}
%   \label{eq:22}
%   \phi(s_\#, t_\#)
%   &=
%     \frac{s_\#}{\xi_\#}\frac{1}{\sqrt{-2 \log s_\#}} \left( 1 + o(1; t_\# \to \infty) \right)
%     \alpha_q \left( \sqrt{2k \xi_\#^2 + 1} + \Phi_{int0}^{(err)}(\xi_\#, \tau_\#) \right) \\
%   &= \alpha_q \frac{s_\#}{\sqrt{|\log s_\#|}}
%     \left( \sqrt{1 + \frac{1}{2k \xi_\#^2}} + \frac{\phi_{int0}^{(err)}(\xi_\#, \tau_\#)}{\sqrt{2k}\xi_\#} \right)
%     (1 + o(1; t_\# \to \infty)) \\
%   &= \alpha_q \frac{s_\#}{\sqrt{|\log s_\#|}} \left(1 + o(1; t_\# \to \infty)\right) \label{phi_sharp_evaluate}
% \end{align}
\begin{align}
  \phi(s, t_\#)
  &=
    \frac{s_\#}{\xi_\#}\frac{1}{\sqrt{-2 \log s_\#}} (1 + o(1; t_\# \to \infty))
    \alpha_q\left(\sqrt{2k \xi^2 + 1 } + \Phi_{int0}^{(err)}\left( \xi, \tau_\# \right) \right) \\
  &= \alpha_q \frac{\xi}{\xi_\#} \frac{s_\#}{\sqrt{|\log s_\#|}}
    \left( \sqrt{1 + \frac{1}{2k\xi^2}} + \frac{\Phi_{int0}^{(err)}(\xi, \tau_\#)}{\sqrt{2k}\xi} \right)(1 + o(1; \tau_\# \to \infty) )
    \\
  &= \alpha_q \frac{s}{\sqrt{|\log s|}}
    \left( \sqrt{ \frac{\log s}{\log s_\#} } \right)
    \left( \sqrt{1 + \frac{1}{2k\xi^2}} + \frac{\Phi_{int0}^{(err)}(\xi, \tau_\#)}{\sqrt{2k}\xi} \right)(1 + o(1; \tau_\# \to \infty) )
    \label{phi_sharp_evaluate}
\end{align}

We assume that the value of $\phi(s_\#, t)$ at $t = t_\#$ is a good approximation for its value at $t = T$. 
\begin{assumption}\label{assume:smalltimeleft}
  \begin{align}
    |\psi(s_\#, t_\#) - \psi(s_\#, T)| = o(1; s_\# \to 0) \psi(s_\#, t_\#),
    \quad
    |\phi(S_\#, t_\#) - \phi(s_\#, T) | = o(1; s_\# \to 0) \phi(s_\#, t_\#).
  \end{align}
\end{assumption}

In particular, Assumption \ref{assume:smalltimeleft} implies the asymptotic profile at the singular time $t = T$:
\begin{align}
  \label{eq:20}
  \psi(s, T) = s(1 + o(1; s \to 0))
  , \quad
  \phi(s, T) = \alpha_q \frac{s}{\sqrt{|\log s|}} (1 + o(1; s \to 0)).
\end{align}
The following lemma shows that Assumption \ref{assume:smalltimeleft} is at least true for the linearization of the system in time.
\begin{lemma}\label{lemma:tderiv_outer}
With all assumptions before Assumption \ref{assume:smalltimeleft}, the time derivatives of $\phi$ and $\psi$ at $(s, t) = (s_\#, t_\#)$ satisfy
  \begin{align}
    \label{eq:19}
    (T-t_\#) \cdot \partial_t \fixs \psi (s_\#, t_\#) = o(1)\psi(s_\#, t_\#),
    \quad
    (T-t_\#) \cdot \partial_t \fixs \phi (s_\#, t_\#) = o(1)\phi(s_\#, t_\#)
  \end{align}
\end{lemma}
\begin{proof}
  The evolution for $\psi$ and $\phi$ (after performing an integration by parts) is
  \begin{align}
  \partial_t \fixs \psi &= \psi_{ss}
                          +
                          \left(
                          p \frac{\psi_s}{\psi}
                          - p \int_0^s \frac{\psi_{ss}}{\psi} ds - q\int_0^s \frac{\phi_{s}^2}{\phi^2} ds
                          \right)\psi_s
                          - \psi^{-1}\psi_{s}^2 - (p-1)\psi^{-1} ,
  \\
  \partial_t \fixs \phi &= \phi_{ss}
                          +
                          \left(
                          p \frac{\psi_s}{\psi}
                          - p \int_0^s \frac{\psi_{ss}}{\psi} ds - q\int_0^s \frac{\phi_{s}^2}{\phi^2} ds
                          \right)\phi_s
                          - \phi^{-1}\phi_{s}^2 - (q-1)\phi^{-1} .
  \end{align}
  To evaluate the terms involving just derivatives of $\phi$ and $\psi$ we can use \eqref{psi_sharp_evaluate} and \eqref{phi_sharp_evaluate}.
  For the nonlocal terms, we need more.  To evaluate the nonlocal term involving $\psi$, note that $\psi \approx s$ is valid in the parabolic region as well.  To evaluate $\int_0^s \frac{\phi_s^2}{\phi^2} ds$, we can apply Cauchy-Schwarz and integrate:
  \begin{align}
    \left| \int_0^s \frac{\phi_s^2}{\phi^2} \right| (s, t_\#)
    &\leq \left( \max_{[0,s]} \left|\phi_s(s, t_\#)\right| \right)  \left| \int_0^s \frac{\phi_s^2}{\phi^2} \right| \\
    &= \left( \max_{[0,s]} \left|\phi_s(s, t_\#)\right| \right)  \left| \frac{1}{\phi(0)} - \frac{1}{\phi(s)} \right| \\
    &\leq 1 \cdot 2 \alpha_q \frac{1}{\alpha_q \sqrt{T-t_\#}}(1 + o(1; t_\# \to \infty))
  \end{align}
\end{proof}
\subsubsection{Conclusion}
Our conjecture is thus as follows:  consider any Ricci flow in the space of metrics we are considering, which has a type-I singularity at $s=0, t=T$ modeled on the standard soliton on $R^{p+1}\times S^q$ and which is either compact or has reasonable growth at infinity.  Then the limit of the metrics as $t \to T$ will have the form
\begin{align}
  \psi &= s (1 + o(1; s \to 0)) \label{final_profile_psi},\\
  \phi &= \alpha_q \sqrt{k} \frac{s}{\sqrt{|\log s|}} \left( 1 + o(1; s \to 0) \right) \label{final_profile_phi}.
\end{align}

This is an unsurprising conclusion if one considers the stability of $\real^{p+1}$ under Ricci flow, and compares with previous results in the $p=0$ case.  In fact the only effect that the value of $p$ has, on the level of our asymptotics, is in on term in the neck region.

\bibliographystyle{amsalpha}
\bibliography{bibliography}

\end{document}